\documentclass[11pt,reqno,tbtags,a4paper]{amsart}
\usepackage{amssymb}
\usepackage[utf8]{inputenc}   % om svenska bokstäver
\usepackage{xpunctuate}
\usepackage{amsbsy,amssymb}
\usepackage{url}
\usepackage{mathtools}
\usepackage{graphicx}
\usepackage[square,numbers]{natbib}
\bibpunct[, ]{[}{]}{;}{n}{,}{,}
\usepackage{color}
\usepackage{bbm}
\usepackage{hyperref}

\title[Unicellular maps vs hyperbolic surfaces in large genus]
{Unicellular maps vs hyperbolic surfaces in large genus: simple closed curves}

\date{23 November, 2021; revised 10 December, 2021}
\author{Svante Janson}
\thanks{Supported by the Knut and Alice Wallenberg Foundation}
\address{Department of Mathematics, Uppsala University, PO Box 480,
SE-751~06 Uppsala, Sweden}
\email{svante.janson@math.uu.se}
\author{Baptiste Louf}
\address{Department of Mathematics, Uppsala University, PO Box 480, SE-751~06 Uppsala, Sweden}
\email{baptiste.louf@math.uu.se}

\subjclass[2010]{} 
\numberwithin{equation}{section}

\allowdisplaybreaks

\theoremstyle{plain}% default
\newtheorem{theorem}{Theorem}[section]
\newtheorem{lemma}[theorem]{Lemma}
\newtheorem{proposition}[theorem]{Proposition}
\newtheorem{corollary}[theorem]{Corollary}

\theoremstyle{definition}

\newtheorem{exampleqqq}[theorem]{Example}

\newtheorem{remarkqqq}[theorem]{Remark}
\newenvironment{remark}{\begin{remarkqqq}}
  {\hfill\qedsymbol\end{remarkqqq}}
\newtheorem{definition}[theorem]{Definition}

\newtheorem*{acks}{Acknowledgements}

\newtheorem{conjecture}[theorem]{Conjecture}
\newtheorem{open}[theorem]{Open problem}

\theoremstyle{remark}

\newenvironment{romenumerate}[1][-10pt]{% optional argument changes indentation
\addtolength{\leftmargini}{#1}\begin{enumerate}% gives (i), (ii) etc.
 \renewcommand{\labelenumi}{\textup{(\roman{enumi})}}%
 }{\end{enumerate}}

\newenvironment{PXenumerate}[1]{%  argument yields prefix
\begin{enumerate}% gives (#1 1), (#1 2) etc.
 \renewcommand{\labelenumi}{\textup{(#1\arabic{enumi})}}%
 }{\end{enumerate}}

\newcounter{oldenumi}
{\setcounter{oldenumi}{\value{enumi}}
\begin{romenumerate} \setcounter{enumi}{\value{oldenumi}}}
{\end{romenumerate}}

\newcounter{thmenumerate}

\newcounter{xenumerate}   %no left indentation; thus wider lines

\newcommand{\refT}[1]{Theorem~\ref{#1}}
\newcommand{\refTs}[1]{Theorems~\ref{#1}}

\newcommand{\refL}[1]{Lemma~\ref{#1}}
\newcommand{\refLs}[1]{Lemmas~\ref{#1}}
\newcommand{\refR}[1]{Remark~\ref{#1}}

\newcommand{\refS}[1]{Section~\ref{#1}}

\newcommand{\refSS}[1]{Section~\ref{#1}}
\newcommand{\refProp}[1]{Proposition~\ref{#1}}

\newcommand{\refD}[1]{Definition~\ref{#1}}

\newcommand\marginal[1]{\marginpar[\raggedleft\tiny #1]{\raggedright\tiny#1}}
\newcommand\REM[1]{{\raggedright\texttt{[#1]}\par\marginal{XXX}}}
\newcommand\XREM[1]{\relax}

\begingroup
  \divide\count255 by 60
  \multiply\count255 by -60
  \advance\count255 by \time
  \ifnum \count255 < 10 \xdef\klockan{\the\count1.0\the\count255}
  \else\xdef\klockan{\the\count1.\the\count255}\fi
\endgroup

\DeclareMathOperator*{\sumx}{\sum\nolimits^{*}}

\DeclareMathOperator*{\hsumx}{\widehat{\sum}}

\newcommand{\sumno}{\sum_{n=0}^\infty}

\newcommand{\sumk}{\sum_{k=1}^\infty}

\newcommand{\sumnu}{\sum_{\nu\ge1}}

\newcommand{\prodir}{\prod_{i=1}^r}

\newcommand{\m}{\mathbf{m}}

\newcommand\set[1]{\ensuremath{\{#1\}}}
\newcommand\bigset[1]{\ensuremath{\bigl\{#1\bigr\}}}
\newcommand\Bigset[1]{\ensuremath{\Bigl\{#1\Bigr\}}}

\newcommand\xpar[1]{(#1)}
\newcommand\bigpar[1]{\bigl(#1\bigr)}
\newcommand\Bigpar[1]{\Bigl(#1\Bigr)}
\newcommand\biggpar[1]{\biggl(#1\biggr)}
\newcommand\lrpar[1]{\left(#1\right)}

\newcommand\bigsqpar[1]{\bigl[#1\bigr]}
\newcommand\Bigsqpar[1]{\Bigl[#1\Bigr]}

\newcommand\lrsqpar[1]{\left[#1\right]}

\newcommand\bigabs[1]{\bigl\lvert#1\bigr\rvert}

\def\rompar(#1){\textup(#1\textup)}    % usage: \rompar(...)
\newcommand\xfrac[2]{#1/#2}
\newcommand\xpfrac[2]{(#1)/#2}

\def\xexp(#1){e^{#1}}
\newcommand\ceil[1]{\lceil#1\rceil}

\newcommand\floor[1]{\lfloor#1\rfloor}

\newcommand\ntoo{\ensuremath{{n\to\infty}}}

\newcommand\Mtoo{\ensuremath{{M\to\infty}}}

\newcommand\start{\text{start}}

\newcommand\eend{\text{end}}

\newcommand\punkt{\xperiod}    % xpunctuate
    
\newcommand\ie{i.e\punkt}
\newcommand\eg{e.g\punkt}

\newcommand\whp{whp}

\newcommand{\tend}{\longrightarrow}
\newcommand\dto{\overset{\mathrm{d}}{\tend}}
\newcommand\pto{\overset{\mathrm{p}}{\tend}}

\newcommand\bbR{\mathbb R}

\newcommand\bbN{\mathbb N}

\newcommand\bbZ{\mathbb Z}

\newcommand\bbZgeo{\mathbb Z_{\ge0}}

\newcommand\E{\operatorname{\mathbb E{}}}
\renewcommand\P{\operatorname{\mathbb P{}}}

\newcommand\Poi{\operatorname{Poi}}

\newcommand\ga{\alpha}
\newcommand\gb{\beta}
\newcommand\gd{\delta}

\newcommand\gl{\lambda}
\newcommand\gL{\Lambda}

\newcommand\gs{\sigma}

\newcommand\cC{\mathcal C}

\newcommand\cE{\mathcal E}

\newcommand\cN{\mathcal N}

\newcommand\cS{{\mathcal S}}
\newcommand\cT{{\mathcal T}}

\newcommand\cX{{\mathcal X}}

\newcommand\bt{\mathbf{t}}

\def\u{\mathbf U_{n,g}}

\newcommand\qw{^{-1}}
\newcommand\qww{^{-2}}
\newcommand\qq{^{1/2}}

\newcommand\intoo{\int_0^\infty}

\newcommand\ooo{[0,\infty)}

\newcommand\dd{\,\mathrm{d}}

\newcommand\lhs{left-hand side}

\newcommand\cperm{C-permutation}
\newcommand\ctree{C-decorated tree}
\newcommand\fS{\mathfrak{S}}
\newcommand\fsc{\fS^{\textsf{C}}}
\newcommand\Nxnm[1]{N_{#1;n,m}}

\newcommand\Nnunm{\Nxnm{\nu}}

\newcommand\Nnu{N_{\nu}}

\newcommand{\CPC}{\cN}

\newcommand\bm{\mathbf{m}}

\newcommand\kkx[1]{^{(#1)}}
\newcommand\kkk{\kkx{k}}

\newcommand\PP{\mathfrak{P}}
\newcommand\PPmx[1]{\PP^{[#1]}}

\newcommand\PPm{\PPmx{\bm}}

\newcommand\PPkmi{\PPmx{m_i}_{k_i}}
\newcommand\PPkm{\PPmx{m}_k}

\newcommand\ppkmi{\pMx{m_i}_{k_i}}
\newcommand\ppkm{\pMx{m}_{k}}
\newcommand\pMx[1]{P^{[#1]}}
\newcommand\pM{\pMx{\bm}}

\newcommand\CC{\mathfrak{C}}
\newcommand\cc{C}
\newcommand\CCx[1]{\CC^{[#1]}}
\newcommand\CCkx[1]{\CCx{#1}_k}

\newcommand\CCkm{\CCkx{m}}

\newcommand\CCkab{\CC_k^{[a,b)}}

\newcommand\tCCx[1]{\widetilde\CC^{[#1]}}
\newcommand\tCCkx[1]{\widetilde\CC^{[#1]}_k}

\newcommand\tCCkm{\tCCkx{m}}
\newcommand\tCCkmi{\tCCx{m_i}_{k_i}}

\newcommand\ccx[1]{\cc^{[#1]}}
\newcommand\cckx[1]{\cc^{[#1]}_k}

\newcommand\cckab{\cc_k^{[a,b)}}

\newcommand\cckm{\cckx{m}}
\newcommand\cckmi{\ccx{m_i}_{k_i}}
\newcommand\cckMx[1]{\cckx{#1;M}}
\newcommand\cckMabx[1]{\cckMx{a^{#1}(M),b^{#1}(M)}}
\newcommand\tccx[1]{\widetilde\cc^{[#1]}}
\newcommand\tcckx[1]{\tccx{#1}_k}

\newcommand\tcckmi{\tccx{m_i}_{k_i}}
\newcommand\tcckm{\tcckx{m}}

\newcommand\cCxy{\cC^{x,y}}
\newcommand\cCoy{\cC^{0,\xmax}}

\newcommand\bP{\mathbf{P}}
\newcommand\Poo{\PP}
 \newcommand\hPoo{\widehat\PP}
\renewcommand\bt{\mathbf t}
\newcommand\simeqx{\equiv}

\newcommand\tpi{\tilde\pi}

\newcommand\Pm{\kappa^{(\bm)}}

\newcommand\mnu{\nu}
\newcommand\gq{\zeta}
\newcommand\cXnm{\cX_{n,m}}
\newcommand\bx{\mathbf{x}}
\newcommand\bga{\boldsymbol{\ga}}

\newcommand\Ext{\operatorname{Ext}}
\newcommand\bxga{\overline{\bx-\bga}}
\newcommand\sfC{\mathsf{C}}
\newcommand\glx{{\widehat\gl}}

\newcommand\CXC{C_0}

\newcommand\gLxx[2]{\gL_{#1}(#2)}

\newcommand\sw{w}
\newcommand\Cat[1]{\mathit{Cat}_{#1}}
\newcommand\summmk{\sum_{|\bm|=m,s(\bm)=k}}
\newcommand\sC{\mathsf{C}}
\newcommand\fC{\mathfrak{C}}
\newcommand\bY{\mathbf Y}
\newcommand\by{\mathbf y}
\newcommand\GG{G}
\newcommand\hgl{\widehat\gl}
\newcommand\Xix{\widehat{\Xi}}

\newcommand\bbbN{\overline{\bbN}}
\newcommand{\sig}{\boldsymbol{\sigma}}
\newcommand{\bgs}{\sig}
\newcommand{\T}{\mathbf{T}}

\newcommand\Lmax{L_n}
\newcommand\Lmaxx{\Lmax}
\newcommand\Linf{L^\bullet}

\newcommand\xmax{y}
\newcommand\Mmax{M}
\newcommand\Interv[1]{[\frac {#1} \Mmax \Lmax,\frac {#1+1} \Mmax
    \Lmax)}

\newcommand\BigInterv[1]{\Bigl[\frac {#1} \Mmax \Lmax,\frac {#1+1} \Mmax
    \Lmax\Bigr)}
\newcommand\setM{\mathcal{M}}
\newcommand{\Bicx}{\operatorname{Bic}}
\newcommand{\Bic}{\mathcal B}

\newcommand{\PU}{\mathcal P}

\begin{document}

\begin{abstract} 

We study uniformly random maps with a single face,  genus $g$, and size
$n$, as $n,g\rightarrow \infty$ with $g=o(n)$, in continuation of several
previous works on the geometric properties of ``high genus maps". 
We calculate the number of short simple cycles, 
and we show convergence of their lengths
(after a well-chosen rescaling of the graph distance)
to a Poisson process,
which happens to be exactly the same as the limit law obtained by Mirzakhani and
Petri (2019)  when they studied simple closed geodesics on random hyperbolic
surfaces under the Weil--Petersson measure as $g\rightarrow \infty$. 

This leads us to conjecture that these two models are somehow ``the same" in
the limit, which would allow to translate problems on hyperbolic surfaces in
terms of random trees, thanks to a powerful bijection of Chapuy, Féray and
Fusy (2013).
\end{abstract}

\maketitle

\section{Introduction}

\subsection{Combinatorial maps.} Maps are defined as gluings of polygons forming a (compact, connected, oriented) surface. They have been studied extensively in the past 60 years, especially in the case of planar maps, i.e., maps of the sphere. They were first approached from the combinatorial point of view, both enumeratively, starting with \cite{Tut63}, and bijectively, starting with \cite{Sch98these}.

More recently, relying on previous combinatorial results, geometric properties of large random maps have been studied. More precisely, one can study the geometry of random maps picked uniformly in certain classes, as their size tends to infinity. In the case of planar maps, this culminated in the identification of two types of ``limits" (for two well defined topologies on the set of planar maps): the local limit (the \emph{UIPT}\footnote{In the case of triangulations, i.e., maps made out of triangles.} \cite{AS03}) and the scaling limit (the \emph{Brownian map} \cite{LG11, Mie11}).

All these works have been extended to maps with a fixed genus $g>0$ \cite{BC86,CMS09,Bet16}.

\subsection{High genus maps.} 
Very recently, another regime has been studied: \emph{high genus maps} are defined as (sequences of) maps whose genus grow linearly in the size of the map. They have a negative average discrete curvature, and can therefore be considered as a discrete model of hyperbolic geometry.
Their geometric properties have been studied, first on unicellular maps \cite{ACCR13,Ray13a,Lou21,SJ358} (i.e., maps with one face), and shortly after on more general models of maps \cite{BL19,BL20,Lou20}.

\subsection{Our results}

While all these works focuses on the regime where $g$ grows linearly in $n$,
we are here interested in the slightly different regime where $g\to\infty$
but $g=o(n)$. We will study the distribution of lengths of simple cycles in unicellular maps (which we studied in the ``linear genus regime" in a previous work \cite{SJ358}). The main interest here is that, with the right rescaling of the graph distance, our result matches exactly a result of Mirzakhani and Petri \cite{MP19} on random hyperbolic surfaces, which leads us to conjecture that these random hyperbolic surfaces can in some sense be approximated by unicellular maps (see Section~\ref{sec_conj} for more details).

Let $\u$ be a uniform unicellular map of genus $g$ and size $n$, and set
\begin{align}\label{ellen}
  \Lmax:=\sqrt{\frac{n}{12 g}},
\end{align}
which will turn out to be the typical
order of the size of the smallest cycles.

\begin{theorem}\label{TPP}
  Suppose that \ntoo{} and that $g=g_n\to\infty$ with $g=o(n)$.
Let $\set{\zeta_i}$ be the set of simple cycles in $\u$, and consider the
(multi)set of their lengths $Z_i:=|\zeta_i|$, scaled as
$\Xi_n:=\bigset{Z_i/\Lmax}
=\bigset{(12g/n)\qq Z_i}$. 
Then the random set\/ $\Xi_n$,
regarded as a point process on $\ooo$, converges in distribution
to a Poisson process on $\ooo$ with intensity
$\xpfrac{\cosh t-1}{t}$.
\end{theorem}

%The scale factor is chosen such that the limiting Poisson process
%is the same as in \citet{MP19}.

For a background on point processes,
see \eg{} \cite[Chapter 12 and 16]{Kallenberg} or
\cite{Kallenberg-rm}.
%and \cite[Section 4]{SJ136}. 
The convergence to a Poisson process in \refT{TPP} can be expressed in
several, equivalent forms.
One equivalent version is the following, 
%by e.g. Kallenberg-rm, Theorems 4.11 and 4.15
%or (not quite) Kallenberg Th 16.16 and Prop 16.17
stated similarly to the main result 
of Mirzakhani and Petri \cite{MP19}.

\begin{theorem}\label{thm_main}
Let $\mathcal C_n^{x,y}$ be the number of simple cycles of\/ $\u$ whose
length belongs 
to $[x\Lmax,y\Lmax]$. 
Then, for every finite set of disjoint intervals 
$[x_1,y_1]$, $[x_2,y_2]$,\dots,$[x_k,y_k]$, the random variables $\mathcal
C_n^{x_i,y_i}$ converge in distribution,
as $\ntoo$, to independent
Poisson variables with parameters $\lambda(x_i,y_i)$ 
where
\begin{equation}\label{lambdaxy}
\lambda(x,y)= \int_{x}^{y} \frac{\cosh t -1}t \dd t.
\end{equation}
\end{theorem}

For comparison, we state the theorem by 
Mirzakhani and Petri \cite{MP19}.%
\footnote{The theorem in \cite{MP19} is stated for primitive closed
  geodesics, but it follows from the proof there that 
whp every primitive closed geodesic with length $\le C$ is simple, 
and thus the same result holds for simple closed geodesics.
The same holds in our \refT{thm_main}, see
\refR{Rprimitive}.}
(See Section~\ref{sec_conj} and the references there for definitions.) 

\begin{theorem}\label{thm_MP}[Mirzakhani--Petri \cite{MP19}]
Let $\mathcal {\widehat {C}}_g^{x,y}$ be the number of simple closed
geodesics in the random hyperbolic surface $\mathbf S_g$ 
%(see Section~\ref{sec_conj} for a definition) 
whose lengths belong
to $[x,y]$. 
Then, for every finite set of disjoint intervals 
$[x_1,y_1], [x_2,y_2],\dots,\allowbreak[x_k,y_k]$, the random variables 
$\mathcal{\widehat{C}}_g^{x_i,y_i}$ converge jointly in distribution, 
as $g\rightarrow\infty$, to independent Poisson variables with parameters $\lambda(x_i,y_i)$
where
$\gl(x,y)$ is given by \eqref{lambdaxy}.
%\begin{equation}
%\lambda(x,y)= \int_{x}^{y} \frac{\cosh t -1}t \dd t.
%\end{equation}
\end{theorem}

Another equivalent version of \refT{TPP} is that if we order the cycles
according to 
increasing length, so that $Z_1\le Z_2\le\dots$, and extend the 
sequence $(Z_1,Z_2,\dots)$ to an infinite one by adding a tail
$\infty,\infty,\dots$, 
then the resulting sequence, after rescaling as above, converges to the
sequence of points in the 
(inhomogeneous) Poisson process defined above, in the usual product
topology on $\ooo^\infty$. (See, e.g., \cite[Lemma 4]{SJ136}.)
In particular, this yields the following corollary,
cf.\ \cite[Theorem 5.1]{MP19}.

\begin{corollary}\label{CPP}
  Let $Z_1^{(n)}$ be the length of the shortest cycle in $\u$.
Then, 
\begin{align}\label{cpp}
%\Bigpar{\frac{12g}{n}}\qq Z_1^{(n)} \dto Z,   
 Z_1^{(n)}/\Lmax \dto Z,   
\end{align}
where $Z$ is a random variable with the
distribution function
\begin{align}
  \P\xpar{Z\le z} = 1-\exp\Bigpar{-\int_0^z \frac{\cosh t-1}{t}\dd t},
\qquad z\ge0.
\end{align}
\end{corollary}

In unicellular maps, all simple cycles are non-contractible; hence
$Z_1^{(n)}$ is the law of the \emph{systole} 
(the size of the smallest non-contractible cycle)
of $\u$.

\subsection{A conjecture}\label{sec_conj}

For $g\geq 2$, there is a natural way of defining a random hyperbolic surface $\mathbf{S}_g$ of genus $g$: the \emph{Weil--Petersson probability measure} (we refer to \cite{MP19} and references therein for more details). It is natural to try to understand the geometric behaviour of these random hyperbolic surfaces as $g\rightarrow\infty$, and this has been done rather extensively in the recent years \cite{GPY11,Mir13,MP19,Mon20,NWX20,Tho20,Wright20,PWX21,MT21,WX21,GLST21,LMS20}.

The similarity between the geometric behaviour of maps and hyperbolic surfaces had been noticed before, but in the precise regime considered in this paper, the ``numerical evidence" provided by Theorems~\ref{thm_main} and~\ref{thm_MP} leads us to conjecture that these two models are somehow ``the same" as $g\to\infty$. We do not know yet what the exact formulation should be, but two features seem to be necessary for this conjecture to be true: first, one should remove the ``fractal part" of the map, and consider the $2$-core of $\mathbf U_{n_g,g}$ instead (see Fig.~\ref{fig_core_decomp}). It is also important to consider $\text{$2$-core}(\mathbf U_{n_g,g})$ as a hyperbolic polygon whose sides were glued, and not just an embedded graph (see Section~\ref{rrrr} for a proper definition of unicellular maps as gluings of a polygon). We can now conjecture the following.

\begin{figure}
\center
\includegraphics[scale=0.5]{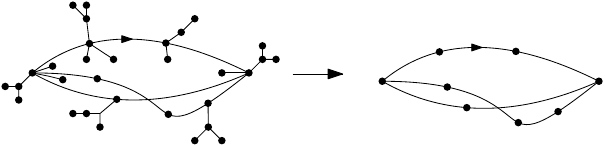}
\caption{A unicellular map and its $2$-core.}\label{fig_core_decomp}
\end{figure}

\begin{conjecture}\label{conj_WP}
Let $n_g$ be such that $g=o(n_g)$ as $g\rightarrow\infty$. 
Then,  $\mathbf U_{n_g,g}$, 
with distances rescaled by the factor $L_{n_g}\qw$, 
and $\mathbf S_g$ can be coupled such that %, as $g\to\infty$,
\begin{align}
  d_{\text{GH}}\left(\text{$2$-core}(\mathbf U_{n_g,g}),\mathbf S_g\right)
\xrightarrow[g\rightarrow\infty]{} 0
\end{align}
in probability,
where $d_{\text{GH}}$ is the Gromov--Hausdorff distance between metric spaces.
\end{conjecture}

There may be small adjustments to make to this conjecture: for instance, maybe a different notion of distance than Gromov--Hausdorff is needed, and on the other hand, we can hope for something stronger that accounts for the topology, e.g. separating curves. Furthermore, there might be a slightly simpler equivalent combinatorial model (for instance, cubic unicellular maps with random edge lengths). We leave these questions open for now.
%Many open questions remain.

We believe the conjecture it is an interesting question for two
reasons. First, it 
would reinforce the ``universality" principle in two-dimensional geometry
(i.e., different models behave the same). And what's more, if this
conjecture is true, any geometric property that we prove on our model of
unicellular maps would hold for hyperbolic surfaces in large genus. But
unicellular maps are easier to work with especially because they are in
bijection with a certain model of trees \cite{CFF13}, see \refSS{SScperm} below.
Therefore, Conjecture~\ref{conj_WP}, if true, would allow us
to transfer any geometric problem
on hyperbolic surfaces onto a problem on \emph{random trees}, which are very
well understood.

There are several open questions remaining for random
hyperbolic surfaces, and the conjecture above (whether true or not) 
suggests studying 
the corresponding problems
on $\mathbf U_{n_g,g}$.
Perhaps the most natural such question is about the diameter:
\begin{open}\label{open_diam}
What is the diameter of $\mathbf U_{n_g,g}$ (rescaled by $L_{n_g}\qw$ as above) ?
\end{open}

For the diameter of random hyperbolic surfaces, a simple area argument gives a
deterministic lower bound of $(1+o(1))\log g$,
%(see \cite{Mag});
while, so far, the best upper bound is $(4+o(1))\log g$ whp,
which can be derived from an
inequality linking the diameter to the spectral gap 
(see \cite{Mag}, combined with \cite{WX21,LW21});
it is natural to try and find the ``right constant" in front of $\log g$.

Several spectral properties of $\mathbf{S}_g$ are still open problems, and might be more tractable on $\mathbf U_{n_g,g}$ (and the associated model of random trees):
\begin{open}\label{open_spectral}
Study the spectral gap, the Cheeger constant and Laplacian eigenfunctions of $\mathbf U_{n_g,g}$.
\end{open}

%\begin{struc}
\subsection{Structure of the paper}
We will end this section with an index of notations, and we will give some
definitions in \refS{sec_def}. In \refS{Sperm}, we prove some results about
\cperm{s}, and in \refS{Spaths}, we study the number of occurrences of paths in
uniformly random trees.  We use these results in \refS{Sctrees} to calculate
the law of cycles in unicellular maps.
%\end{struc}

\begin{acks}

We are grateful to Bram Petri and Stephan Wagner for enlightening discussions. We also thank Thomas Budzinski, Nicolas Curien and Yunhui Wu for comments on the first version of this paper.
\end{acks}

\subsection*{Index of notations}
(Not including some that are only used locally.)
\begin{itemize}
\item $g=g_n$: the genus of the map. (We assume $1\ll g_n \ll n$.)
\item $\u$: a uniformly random unicellular map of genus $g$ and size $n$.
\item $\mathbf T=\mathbf T_n$: a uniformly random tree of size $n$.
\item $(T_n)_{n\geq 1}$:  a deterministic sequence of trees.
\item   $\Lmax:=\sqrt{\frac{n}{12 g}}$: the scaling factor for the graph
  distance. 
\item $\Mmax$: a large integer. (Usually fixed.)
%\item $\Lmax=\Lmax_n:=\xmax \Lmax$.
\item $\fsc_{n,m}$: the set of \cperm{s} on $n$ elements and $m$  cycles.
\item $\sig$: a uniformly random element in $\fsc_{n+1,n+1-2g}$. (Depends thus implicitly on $n$.)
\item $T$:  a fixed rooted tree.
\item $\bt$:  a fixed rooted tree.
\item $N_\bt(T)$: the number of occurrences of $\bt$ in $T$.
\item $P_i(T)$: the number of paths of length $\ell\in\Interv{i}$ in $T$.
\item $\bm$:  a finite sequence of non-negative integers.
\item $\bP$:  a list of pairwise (vertex) disjoint paths.
\item $s(\bP)$: the number of paths in $\bP$.
\item $\ell(\bP)$: the total length of the paths in $\bP$.
\item $\Poo(T)$, $\PPm(T)$, $\PPkm(T)$
$\CCkm(T,\gs)$, $\CCkab(T,\gs)$, $\tCCkm(T,\gs)$
: 
 sets of lists of  disjoint paths in $T$.
\item $\pM(T)$, $\ppkm(T)$,
$\cckm(T,\gs)$, $\cckab(T,\gs)$, $\tcckm(T,\gs)$:
cardinalities of these sets.

%%\item $C_{n,g_n}^{(k)}$: the number of simple
%cycles of length $k$ in a
%uniformly random unicellular map of size $n$ and genus $g_n$ (alternatively, in the underlying graph of $(\T,\sig)$);
\item $\Pm$: the constant \eqref{eq_Pm}.
\end{itemize}

\section{Definitions and notations}\label{sec_def}
\subsection{Parameters} \label{SSparam}
We will discretize our problem in order to be able
to reason on a finite number of quantities.
For most of the proof,
we will fix a (large) integer $\Mmax>0$.
Only in \refSS{SSMoo} we will let $M\to\infty$, which eventually will yield
our final results. For notational convenience, we will usually omit $n$ and
$M$ from the notation when there is no risk of confusion, but it should be
remembered that most variables introduced below depend on both $n$ and $M$.

Recall that $\Lmax$ was defined in \eqref{ellen}.
Note that, by our assumptions on $g=g_n$, we have $\Lmax\to\infty$ and
$\Lmax=o\bigpar{n\qq}$. 
We define also
\begin{align}\label{Linf}
\Linf&:= (\log g) \Lmax.
\end{align}
The exact definition is not important; we will only use the properties
$\Lmax \ll \Linf \ll n\qq$ as \ntoo.

\subsection{Paths, cycles and trees}

By a path $p$, we mean a simple path, i.e., a list of $\ell+1$ distinct vertices
$v_0,\dots,v_\ell$ and $\ell\ge1$ edges $v_{i-1}v_i$, 
where $\ell$ is the \emph{length} or $\emph{size}$ of the path, denoted
$|p|$. (Note that we require $|p|>0$.)
All our paths are \emph{oriented}, i.e., they have a start $\start(p)=v_0$
and an end $\eend(p)=v_\ell$, which together are the \emph{endpoints}
$\Ext(p):=\set{\start(p),\eend(p)}$.

Similarly, a cycle  means  a simple cycle, i.e., a set of $\ell$ distinct
vertices $v_1,\dots,v_\ell$ and $\ell\ge2$ edges $v_iv_{i+1}$ (where $v_{\ell+1}$
is interpreted as $v_1$), where $\ell$ is the \emph{length} or
\emph{size} of the cycle.
Our cycles are {unoriented}, and they do not
have any designated starting point; thus the vertices 
$v_1,\dots,v_\ell$ can be ordered in $2\ell$ different ways yielding the
same cycle.

Our trees will be plane trees, i.e., trees embedded in the plane (up to
obvious isomorphism). The size $|\bt|$ of a tree $\bt$ is its number of edges.
At each vertex $v$ of $\bt$, the gaps between two adjacent edges are called
\emph{corners}; thus, there are $d$ corners at a vertex of degree $d$, and
hence in total $2|\bt|$ corners in a tree $\bt$.

Our trees are usually rooted; the root of a tree is a corner.
(This is equivalent to the slightly different definition of rooted plane trees
in \eg{} \cite[Section 1.1.2]{Drmota}.)
%Note that this is equivalent to the standard definition of a rooted plane
%tree (or rooted ordered tree) as a tree rooted at a vertex, and with an
%ordering of the children of 
%each vertex, see \eg{} \cite[Section 1.1.2]{Drmota}.
%(In this equivalence, 
%given a root corner, we order the edges from the root clockwise,
%starting with the first edge after the root corner.
%Equivalently, we may regard the tree as a planted plane tree
%\cite[Section 1.2.2]{Drmota} with a phantom root in the root corner.)
%Note that a rooted plane tree can be regarded as labelled, since there are
%no automorhisms (preserving the root) except the identity.

We emphasize that the size of a path, cycle or tree is its number of edges.

Let $T$ be a rooted tree. For any tree $\bt$, let $N_\bt(T)$ be the number
of occurrences of $\bt$ in $T$.
Furthermore, let
\begin{equation}\label{Pit}
P_i(T):=\sum_\bt N_\bt(T),
\qquad i\ge0,
\end{equation}
where the sum spans over all paths of size belonging to $\Interv{i}$.
(See \refSS{SSparam} for the (implicit) parameter $M$ and $\Lmax$.)

We denote by $\cT_n$ the set of rooted plane trees of size $n$, and by
$\T=\mathbf{T}_n$ a uniformly random element of $\cT_n$.

\subsection{Lists of paths}

Given a rooted plane tree $T$, 
let $\Poo(T)$ be the set of all lists $\bP=(p_1,\dots,p_k)$ of pairwise
vertex
disjoint paths in $T$, of arbitrary length $k\ge1$.
For a list $\bP=(p_1,\dots,p_k)\in\Poo(T)$, let $s(\bP):=k$, the number of
paths in the list, and $\ell(\bP):=\sum_1^k|p_i|$, their total length.
%(Recall that $|p_i|$ is the number of edges in $p_i$.)
Also, let
$\Ext(\bP):=\bigcup_i\Ext(p_i)=\set{\start(p_i),\eend(p_i):i=1,\dots,k}$, the
set of endpoints of the paths in $\bP$; note that $|\Ext(\bP)|=2s(\bP)$
since the paths are disjoint.

%For an integer $\ell\ge1$, let
%\begin{align}\label{PPl}
%\Poo(T):=\set{\bP\in\Poo(T):\ell(\bP)=\ell},   
%\end{align}
%the set of lists of paths with
%total length $\ell$.

Furthermore, 
let $\setM$ be the set of all (non-empty)
finite sequences of non-negative integers.
If
$\mathbf{m}=(m_1,\ldots,m_k)\in\setM$,
we write $|\mathbf{m}|=m_1+m_2+\ldots+m_k$ and $s(\mathbf{m})=k\ge1$. 
We define
\begin{align}\label{PPm}
\PPm(T)&:=
\Bigset{\bP=(p_1,\dots,p_k)\in\Poo(T): |p_i|\in\BigInterv{m_i} \, \forall i},
\\%\intertext{and}
  \PPkm(T)&:=\bigcup_{|\m|=m\atop s(\m)=k} \PPm(T). \label{PPkm}
\end{align}
We let $\pM(T):=|\PPm(T)|$ 
and
$\ppkm(T):=|\PPkm(T)|$
be the cardinalities of these sets of lists.

Note that it follows from the definition \eqref{PPm} that if
$\bP\in\PPm(T)$, then
\begin{align}\label{lp}
\frac{|\bm|}{\Mmax}\Lmax
\le  \ell(\bP) \le \frac{|\bm|+s(\bm)}{\Mmax}\Lmax.
\end{align}

Define also, for
$\mathbf{m}=(m_1,\ldots,m_k)\in\setM$ as above,
\begin{equation}\label{eq_Pm}
\Pm=\prod_{i=1}^{s(\mathbf{m})} (2m_i+1).
\end{equation}

\subsection{Unicellular maps}\label{rrrr}
A unicellular map of \emph{size} $n$ is a $2n$-gon whose sides were glued two by
two to form a (compact, connected, oriented) surface. The \emph{genus} of
the map is the genus of the surface created by the gluings (its number of
handles). After the gluing, the sides of the polygon become the \emph{edges}
of the map, and the vertices of the polygon become the \emph{vertices} of
the map. 
Note that the number of edges equals the size $n$.
By Euler's formula, a unicellular map of genus $g$ and size $n$ has $n+1-2g$
vertices.
As for trees, the gaps between two adjacent (half-)edges around a vertex are called
\emph{corners}, and there are $2n$ corners in a unicellular map of size $n$.
 The underlying graph of a unicellular map is the graph obtained
from this map by only remembering its edges and vertices.
(In general, this is a multigraph.)

We consider in this paper only
\emph{rooted} unicellular maps, where a corner is marked as the
\emph{root}. 
The underlying graph is then a rooted graph.

A rooted unicellular map of genus $0$ is the same as a rooted plane tree.

We denote by $\u$ a uniformly random unicellular map of size $n$ and genus $g$.

\subsection{\cperm{s} and \ctree{s}}\label{SScperm}

A \emph{\cperm} is a permutation whose cycles are of odd length.
Let $\fsc_n$ be the set of \cperm{s} of length $n$, and $\fsc_{n,m}$ the
subset of permutations in $\fsc_n$ with exactly $m$ cycles.
(This is empty unless $n\equiv m \pmod 2$; we assume tacitly in the sequel
that we only consider cases with $\fsc_{n,m}\neq\emptyset$.)
Note that our definition of a \cperm{} differs from the one given in \cite{CFF13}, where each cycle carries an additional sign. Here we do not include the signs as they will not play a role in our proofs.

A \emph{\ctree{}} of size $n$ and genus $g$ is a pair $(T,\sigma)\in
\cT_n\times \fsc_{n+1,n+1-2g}$ where $\sigma$ is seen as a \cperm{} of the
vertices of $T$ (given an arbitrary labeling of the vertices of $T$, 
for example the one given by a depth first search
with left to right child ordering). The underlying graph of $(T,\sigma)$ is the graph obtained by merging the vertices of $T$ that belong to the same cycle in $\sigma$. If $v,v'\in T$, we write $v\sim v'$ if $v$ and $v'$ belong to the same cycle in $\sigma$.

\begin{theorem}[\cite{CFF13}, Theorem 5]\label{thm_ctrees}
Unicellular maps of size $n$ and genus $g$ are in $2^{2g}$ to $1$ correspondence  with \ctree{s} of size $n$ and genus $g$. This correspondence preserves the underlying graph.
\end{theorem}

Therefore, with this correspondence, it is sufficient to study \ctree{s}.

\subsection{Further notation}

We let $(n)_r$ denote the descending factorial $n(n-1)\dotsm(n-r+1)$.

For a real number $x$, let $(x)_+:=x\vee0:=\max\set{x,0}$.

$\Poi(\gl)$ denotes the Poisson distribution with parameter $\gl$.

Convergence in distribution and in probability are denoted $\dto$
and $\pto$, respectively.

whp means with probability $1-o(1)$ as \ntoo.

Unspecified limits are as \ntoo.

\section{Cycles in \cperm{s}}\label{Sperm}

In this section, we give several lemmas on cycles in random \cperm{s};
the only results that are used outside this section are Lemmas~\ref{LXC}
and~\ref{LXD}. 

We will use $\nu$
as an index denoting cycle lengths in a \cperm.
Recall that only cycles of odd lengths are allowed; thus
it is tacitly understood that $\nu$ ranges over the odd natural numbers 
(or a subset of them if indicated).
(The same applies to $\nu_i$ and $\mu$.)

Let $n$ and $g$ be given, and 
let $m:=n-2g$.
Let $\bgs=\bgs_{n,m}$ be a uniformly random element of $\fsc_{n,m}$, 
and let $\Nnu=\Nnunm$ be its number of cycles of length $\nu$.
%Thus $\Nnunm=0$ unless $\nu\ge1$ is odd.
%Note also that
%\begin{align}\label{eq_constraint_cycles}
% \sum_i \Nnunm = m=n-2g, &&&
%\sum_i i\Nnunm = n.
%\end{align}

Assume that $\bx=(x_1,x_3,\dots)$ is a sequence of non-negative integers, 
with only finitely many $x_\nu\neq0$.
Let $\CPC(\bx)$  %$\CPC(x_1,x_3,\dots)$ 
be the number of \cperm{s} with exactly $x_\nu$ cycles
of size $\nu$ for every $\nu\ge1$. 
%(Recall that only cycles of odd lengths are allowed. All sums and products
%below over cycle lengths $i$ are tacitly over odd $i$ only.
Recall that these permutations belong to $\fsc_{n,m}$ if and only if
\begin{align}\label{xc1}
 \sumnu x_\nu &= m=n-2g, &
\sumnu \nu x_\nu &= n
,\end{align}
which imply
\begin{align}
  \label{xc13}
x_3&=\tfrac12\Bigpar{n-(n-2g)-\sum_{\nu\ge5}(\nu-1)x_\nu}
=g-\sum_{\nu\ge5}\frac{\nu-1}2x_\nu,
\\\label{xc11}
x_1&=n-3x_3-\sum_{\nu\ge5}\nu x_\nu
=n-3g+\sum_{\nu\ge5}\frac{\nu-3}2x_\nu.
\end{align}

Fix $n$ and $m$ and 
let $\cXnm$ be the set of all non-negative integer sequences 
$\bx=(x_1,x_3,\dots)$ 
that satisfy \eqref{xc1}, and thus \eqref{xc13}--\eqref{xc11}.
If $\bx\in\cXnm$, 
it is easily shown that  
\begin{equation}\label{xc2}
\CPC(\bx)
%\CPC(x_1,x_3,\dots)
=\frac{n!}{\prod_{\nu\geq 1}x_\nu!\,\nu^{x_\nu}}
.\end{equation}
For $\bx\in\cXnm$,
let also
\begin{align}\label{xcp}
% p(x_1,x_3,\dots)
p(\bx)  
:=\P\bigpar{\Nnunm=x_\nu,\;\forall \nu}
=\frac{\CPC(\bx)}{|\fsc_{m,n}|}.
\end{align}

\begin{lemma}\label{LXA}
  Suppose that $g < n/3$.  Let $\bx\in\cXnm$,  let $\mu=2k+1\ge5$ be odd
and assume $x_\mu>0$.
Let $\bx'$ be given by
  \begin{align}\label{lxa1}
    x'_\nu=
    \begin{cases}
      x_1-\frac{\mu-3}2,&\nu=1,\\
      x_3+\frac{\mu-1}2,&\nu=3,\\
x_{\nu}-\gd_{\nu,\mu},& \nu\ge5.
    \end{cases}
  \end{align}
Then $\bx'\in\cXnm$ and
\begin{align}\label{lxa2}
  x_\mu p(\bx) \le \frac{(3g)^k}{\mu(n-3g)^{k-1}}p(\bx')
=\frac{(3g)^{(\mu-1)/2}}{\mu(n-3g)^{(\mu-3)/2}}p(\bx').
\end{align}
\end{lemma}

\begin{proof}
Note that $x'_1$ and $x'_3$ are defined such that
  \eqref{xc13}--\eqref{xc11} hold for $\bx'$. 
Furthermore, $x'_\mu=x_\mu-1\ge0$ by assumption,
and thus also, by \eqref{xc11} and the assumption $g < n/3$,
\begin{align}\label{lxa3}
    x'_1=n-3g+\sum_{\nu\ge5}\frac{\nu-3}2x'_\nu
\ge n-3g \ge0
\end{align}
Hence, $x'_\nu\ge0$ for all $\nu$, and thus $\bx'\in\cXnm$.

%Thus $x_1'=x_1-(k-1)$ and $x_3'=x_3+k$.
Note that $x_1=x'_1+k-1$ and $x_3=x'_3-k$, 
and also  that $x_1'\ge n-3g$ by \eqref{lxa3} and $x_3'\le g$ by \eqref{xc13}.
Hence,  \eqref{xcp} and \eqref{xc2} yield
\begin{align}
  \frac{p(\bx)}{p(\bx')}
&=
  \frac{\CPC(\bx)}{\CPC(\bx')}
=\frac{\prod_{\nu\geq 1}x'_\nu!\,\nu^{x'_\nu}}{\prod_{\nu\geq  1}x_\nu!\,\nu^{x_\nu}}
\notag\\&
=\frac{x_1'!\,x'_3!\, 3^{x'_3} (x_\mu-1)!\,\mu^{x_\mu-1}}
{(x'_1+k-1)!\,(x'_3-k)!\,3^{x'_3-k}x_\mu!\,\mu^{x_\mu}}
\notag\\&
\le \frac{(x_3')^k 3^k}{(x_1')^{k-1}x_\mu \mu}
\le \frac{g^k 3^k}{(n-3g)^{k-1}x_\mu \mu}
.\end{align}
The result \eqref{lxa2} follows.
\end{proof}
%We can also find a similar lower estimate, but we will not need that.

We can now easily estimate the mean 
of $\Nnunm$
as well as higher (mixed) factorial moments.
\begin{lemma}\label{LXB}
  Suppose that $g <n/3$.  
Then, for every $\nu\ge3$,
\begin{align}\label{lxb1}
  \E \Nnunm \le \gl_\nu:=
\frac{(3g)^{(\nu-1)/2}}{\nu(n-3g)^{(\nu-3)/2}}.
\end{align}
More generally, 
%for factorial moments,
for any sequence $(\ga_\nu)_3^\infty$ of non-negative
integers (with only finitely many non-zero), 
\begin{align}\label{lxb2}
  \E \Bigsqpar{\prod_{\nu\ge3} (\Nnunm)_{\ga_\nu}}
\le\prod_{\nu\ge3}\gl_\nu^{\ga_\nu}.
\end{align}
\end{lemma}
\begin{proof}
Let $\mu=2k+1\ge5$. 
For $\bx\in\cXnm$ with $x_\mu\ge1$, define
$\bx'$ as in \eqref{lxa1} and note that with $\gl_\mu$ defined in
\eqref{lxb1}, \eqref{lxa2} says 
\begin{align}\label{bach}
x_\mu p(\bx)\le \gl_\mu p(\bx').  
\end{align}
Hence, \refL{LXA} implies,
noting that
the map $\bx\mapsto\bx'$ is injective,
\begin{align}\label{lxb4}
  \E\Nnunm&
=\sum_{\bx\in\cXnm} x_\mu p(\bx)
=\sum_{\bx\in\cXnm,\,x_\mu\ge1} x_\mu p(\bx)
\notag\\&
\le \sum_{\bx\in\cXnm,\,x_\mu\ge1} \gl_\mu p(\bx')
\le \gl_\mu\sum_{\bx'\in\cXnm} p(\bx')
%\notag\\&
= \gl_\mu.
\end{align}
This shows \eqref{lxb1} for $\nu=\mu\ge5$.
For $\nu=3$, we simply note that by \eqref{xc13},
\begin{align}\label{lxb3}
  \Nxnm3\le g = \gl_3.
\end{align}

We prove \eqref{lxb2} similarly. Suppose first that $\ga_3=0$, and let
$\bga=(0,0,\ga_5,\ga_7,\dots)$.
If $\bx\in\cXnm$ and $x_\nu\ge \ga_\nu$ for all $\nu$, 
let $\bxga$ denote the element in $\cXnm$ 
with coordinates $x_\nu-\ga_\nu$ for $\nu\ge5$ 
(and for $\nu=1,3$,  given from these by
\eqref{xc13}--\eqref{xc11}). 
Then 
repeated use of \eqref{bach} yields
\begin{align}\label{lxb5}
p(\bx)  \prod_{\nu\ge5}(x_\nu)_{\ga_\nu} 
\le \prod_{\nu\ge5} \gl_\nu^{\ga_\nu} \cdot p\bigpar{\bxga}.
\end{align}
Hence,
\begin{align}\label{lxb6}
  \E \Bigsqpar{\prod_{\nu\ge5} (\Nnunm)_{\ga_\nu}}
&=\sum_{\bx\ge\bga} p(\bx)  \prod_{\nu\ge5}(x_\nu)_{\ga_\nu} 
\le \sum_{\bx\ge\bga}\prod_{\nu\ge5} \gl_\nu^{\ga_\nu} \cdot p\bigpar{\bxga}
\notag\\&
\le\prod_{\nu\ge5}\gl_\nu^{\ga_\nu}.
\end{align}
This proves \eqref{lxb2} when $\ga_3=0$. The general case follows by this and
the deterministic bound \eqref{lxb3}.
\end{proof}

The estimate in \refL{LXB} shows that in our range $g=o(n)$,
cycles of length 5 or more are few, and 
we will see in results and proofs below that they are
insignificant for our purposes.
The estimate in \refL{LXB} seems to be rather sharp for all $\nu$ that are
not very large, but we will not study this further. We give only a matching
lower bound in the case $\nu=3$.

\begin{lemma}  \label{LX3}
If $g<n/6$, then
\begin{align}\label{lx3}
  g-\frac{5g^2}{n-6g} \le \E\Nxnm3 \le g
%g \ge \E\Nxnm3 \ge  g-\frac{5g^2}{n-6g}
.\end{align}
In particular, as \ntoo{} with $g=o(n)$, 
$\E\Nxnm3\sim g$.
\end{lemma}

\begin{proof}
  By \eqref{xc13} and \refL{LXB},
  \begin{align}
\E \bigpar{g-\Nxnm3}&
= \E\sum_{\nu\ge5}\frac{\nu-1}2\Nnunm     
\le\sum_{\nu\ge5}\frac{\nu-1}2\gl_\nu
\notag\\&
=\sum_{k\ge2} \frac{(3g)^k}{2(n-3g)^{k-1}}
=\frac{(3g)^2}{2(n-3g)}\Bigpar{1-\frac{3g}{n-3g}}\qw
\notag\\&
\le\frac{5g^2}{n-6g}.
  \end{align}
The lower bound in \eqref{lx3} follows. The upper bound follows trivially
from the deterministic bound \eqref{lxb3}.
\end{proof}

\begin{lemma}\label{LXC}
Let $\cE_{n,g}^{(r)}$ be the event where, in $\sig$, $2i-1$ and $2i$
belong to the same cycle for all $1\leq i\leq r$, and  these $r$ cycles
are distinct, and let $\P_{n,g}^{(r)}=\P(\cE_{n,g}^{(r)})$. 
If\/ $g\le n/7$, then for all $r\le n/2$,
\begin{equation}\label{lxc1}
\P_{n,g}^{(r)}
\le \frac{1}{(n)_{2r}}\gL^r
\le \Bigpar{\frac{\CXC g}{n^2}}^r  ,
\end{equation}
where $\CXC<2200$ is an absolute constant and
\begin{align}\label{lxc2}
\gL  =6g\Bigpar{\frac{1-3g/n}{1-6g/n}}^2.
\end{align}
Moreover, as \ntoo{} with $g\to\infty$ and $g=o(n)$, 
$\gL\sim 6g$ and, for any fixed $r\ge1$,
\begin{align}\label{lxc3}
\P_{n,g}^{(r)}
\sim \Bigpar{\frac{6g}{n^2}}^r  
.\end{align}
\end{lemma}

By symmetry, the probability is the same if we replace the nodes
$1,\dots,2r$ by any other $2r$ fixed nodes in $[n]$.

\begin{proof}
  Let $\tau$ be a uniformly random permutation of $[n]$ independent of $\bgs$.
By the invariance just mentioned, the probability is the same if we instead
consider the nodes $\tau(1),\dots,\tau(2r)$, which will be convenient in the
proof. 

For a sequence $\nu_1,\dots,\nu_r$, let $\cE(\nu_1,\dots,\nu_r)$ be the
event that there are distinct cycles $\sfC_1,\dots,\sfC_r$ in $\bgs$ such
that
$|\sfC_i|=\nu_i$ and $\tau(2i-1),\tau(2i)\in\sfC_i$ for every $i\le r$.
We may assume $\nu_i\ge3$ for every $i$, 
since otherwise the event is impossible. 
%since otherwise the probability is 0.

We first compute the conditional probability
$\P\bigpar{\cE(\nu_1,\dots,\nu_r)\mid\bgs}$.
Let $\ga_\nu:=|\set{i:\nu_i=\nu}|$,
the number of the cycles $\sfC_i$ that are required to have length $\nu$.
Given $\bgs$, there are $\Nnunm$ cycles of length $\nu$, and thus
$\prod_\nu (\Nnunm)_{\ga_\nu}$ ways to choose the cycles
$\sfC_1,\dots,\sfC_r$.
Given these cycles, the probability that $\tau(2i-1),\tau(2i)\in\sfC_i$ for
all $i$ is
\begin{align}\label{lxc6}
  \frac{1}{(n)_{2r}}\prod_{i=1}^r\bigpar{\nu_i(\nu_i-1)}
=  \frac{1}{(n)_{2r}}\prod_{\nu}\bigpar{\nu(\nu-1)}^{\ga_\nu}.
\end{align}
Hence, 
\begin{align}\label{lxc7}
  \P\bigpar{\cE(\nu_1,\dots,\nu_r)\mid\bgs}
=\prod_\nu (\Nnunm)_{\ga_\nu}\cdot
 \frac{1}{(n)_{2r}}\prod_{\nu}\bigpar{\nu(\nu-1)}^{\ga_\nu}.
\end{align}
Define, for $\nu\ge3$,
\begin{align}\label{glx}
  \glx_\nu:=\nu(\nu-1)\gl_\nu 
=(\nu-1)\frac{(3g)^{(\nu-1)/2}}{(n-3g)^{(\nu-3)/2}}.
\end{align}
Taking the expectation in \eqref{lxc7}, we obtain by \refL{LXB},
\begin{align}\label{lxc8}
 \P\bigpar{\cE(\nu_1,\dots,\nu_r)}
&=  \E  \P\bigpar{\cE(\nu_1,\dots,\nu_r)\mid\bgs}
\notag\\&
= \frac{1}{(n)_{2r}}\prod_{\nu}\bigpar{\nu(\nu-1)}^{\ga_\nu}
\cdot\E \prod_\nu (\Nnunm)_{\ga_\nu}
\notag\\&
\le\frac{1}{(n)_{2r}}\prod_{\nu}\bigpar{\nu(\nu-1)}^{\ga_\nu}
\cdot\prod_\nu\gl_\nu^{\ga_\nu}
\notag\\&
=\frac{1}{(n)_{2r}}\prod_{\nu}{{\glx_\nu}}^{\ga_\nu}
%\notag\\&
=\frac{1}{(n)_{2r}}\prod_{i=1}^r\glx_{\nu_i}
.\end{align}
Summing over all $\nu_1,\dots,\nu_r$ yields the result
\begin{align}\label{lxc9}
  \P_{n,g}^{(r)}
&=\sum_{\nu_1,\dots,\nu_r}  \P\bigpar{\cE(\nu_1,\dots,\nu_r)}
\le
\frac{1}{(n)_{2r}}
\sum_{\nu_1,\dots,\nu_r}\prod_{i=1}^r\glx_{\nu_i}
=
\frac{1}{(n)_{2r}}
\biggpar{\sum_{\nu\ge3}\glx_\nu}^r.
\end{align}
We have
\begin{align}\label{gL2}
\sum_{\nu\ge3}\glx_\nu
&=\sum_{\nu\ge3}(\nu-1)\frac{(3g)^{(\nu-1)/2}}{(n-3g)^{(\nu-3)/2}}
=\sum_{k\ge1}2k\frac{(3g)^{k}}{(n-3g)^{k-1}}
\notag\\&
= 6g\Bigpar{1-\frac{3g}{n-3g}}\qww
=6g\Bigpar{\frac{1-3g/n}{1-6g/n}}^2
=\gL,
\end{align}
as defined in \eqref{lxc2}.
Hence, \eqref{lxc9} proves the first inequality in  \eqref{lxc1}.

Moreover, using Stirling's formula,
\begin{align}
  (n)_{2r}^{1/2r} \ge (n!)^{1/n} \ge n/e
\end{align}
and if $g/n\le 1/7$, then $\gL\le 6g(1-6/7)^{-2}=294 g$. 
Hence, the second inequality in \eqref{lxc1} holds with 
$\CXC=294 e^2 < 2200$. % 2172.38

For a fixed $r$, we have $\gL\sim 6g$ since $g/n\to0$, and thus 
\eqref{lxc1}  yields the (implicit) upper bound in \eqref{lxc3}.
For a matching lower bound,
we consider only the case $\nu_1=\dots=\nu_r=3$.
We have, by again taking the expectation in \eqref{lxc7},
\begin{align}\label{lxc10}
\P_{n,g}^{(r)}
\ge  \P\bigpar{\cE(3,\dots,3)}
= \frac{1}{(n)_{2r}} 6^r \E (\Nxnm3)_r.
\end{align}
Furthermore, by Jensen's inequality and \refL{LX3},
for any fixed $r$,
\begin{align}\label{lxc11}
  \E (\Nxnm3)_r
&\ge \E (\Nxnm3-r)_+^r
\ge  (\E \Nxnm3-r)_+^r
=\bigpar{ g+O(g^2/n)+O(1)}^r
\notag\\&
%=g^r \bigpar{1+O(g/n)+O(1/g)}
\sim g^r.
\end{align}
The (implicit) lower bound in \eqref{lxc3} 
follows from \eqref{lxc10} and \eqref{lxc11}.
\end{proof}

In \refL{LXC}, each of the cycles that contain one of the distinguished
points $1,\dots,2k$ contains exactly two of them.
We will also use an estimate of the probability that some cycle contains more
than two of the distinguished points.

\begin{lemma}
  \label{LXD}
Assume\/ $g\le n/7$.
For every fixed $k\ge1$ and $r\le k/2$, the probability that $1,\dots,k$
belong to exactly $r$ different cycles in $\bgs$, with at least two of these
points in each cycle, is, for some constants $C=C(k)$,
\begin{align}\label{lxd}
  \le C g^r/(n)_k
\le C g^r/n^k.
\end{align}
\end{lemma}
(For $r=k/2$, this is a just weaker version of \refL{LXC}.)

\begin{proof}
We must have $n\ge k$, and thus the two bounds in \eqref{lxd} are equivalent
(with different $C$).

We argue similarly to the proof of \refL{LXC}.
 We use again randomization, and let $\tau$ be a
  random permutation of $[n]$ independent of $\bgs$.
For sequences $\nu_1,\dots,\nu_r$ and $\gb_1,\dots,\gb_r$,
let $\cE(\nu_1,\dots,\nu_r;\gb_1,\dots\gb_r)$ be 
the event that there are distinct cycles $\sfC_1,\dots,\sfC_r$ in $\bgs$ such
that
$|\sfC_i|=\nu_i$ and exactly $\gb_i$ of the points $\tau(1),\dots,\tau(k)$
belong to $\sfC_i$ for every $i\le r$.
We assume $\gb_i\ge2$ and $\sum_i\gb_i=k$, and 
we also assume $\nu_i\ge3$ for every $i$, 
since otherwise the event is impossible.  

As in the proof of \refL{LXC},
let $\ga_\nu:=|\set{i:\nu_i=\nu}|$. Then, again,
given $\bgs$, there are 
%$\Nnunm$ cycles of length $\nu$, and thus
$\prod_\nu (\Nnunm)_{\ga_\nu}$ ways to choose the cycles
$\sfC_1,\dots,\sfC_r$.
Given these cycles, 
the conditional probability that $\gb_i$ of $\tau(1),\dots,\tau(k)$ belong
to $\sfC_i$, $i=1,\dots,r$, is
\begin{align}\label{lxd1}
 \le C \frac{1}{(n)_k}\prodir (\nu_i)_{\gb_i}. 
\end{align}
(In this proof, $C$ denotes constants that may depend on $k$ but not on
other variables. In \eqref{lxd1} we may take $C=k!$.)
Hence, using \refL{LXB},
\begin{align}\label{lxd2}
&\P\bigsqpar{\cE(\nu_1,\dots,\nu_r;\gb_1,\dots\gb_r) }  
\le \frac{C}{(n)_k}\prodir (\nu_i)_{\gb_i}
\E \prod_{\nu} (\Nnunm)_{\ga_\nu}
\notag\\&\qquad
\le  \frac{C}{(n)_k}\prodir (\nu_i)_{\gb_i}
 \cdot\prod_{\nu} \gl_\nu^{\ga_\nu}
= \frac{C}{(n)_k}\prodir (\nu_i)_{\gb_i}\gl_{\nu_i}
.\end{align}
We sum \eqref{lxd2} first over all $\nu_i,\dots,\nu_r$, and note that by
\eqref{lxb1} and the assumption $g\le n/7$, for every fixed $\gb\ge2$,
\begin{align}\label{lxd3}
  \sum_{\nu\ge3} (\nu)_\gb \gl_\nu \le C(\gb) g. 
\end{align}
Since we only consider $\gb_i\le k$, 
we thus obtain from \eqref{lxd2},
\begin{align}\label{lxd4}
&\sum_{\nu_1,\dots,\nu_r}\P\bigsqpar{\cE(\nu_1,\dots,\nu_r;\gb_1,\dots\gb_r) }  
\le \frac{C}{(n)_k}\prodir \sum_{\nu_i\ge3} (\nu_i)_{\gb_i}\gl_{\nu_i}
\le \frac{C}{(n)_k} g^r
.\end{align}
The result follows by summing over the $O(1)$ allowed $(\gb_1,\dots,\gb_r)$.
\end{proof}

\section{Counting paths in trees}\label{Spaths}

As in Section~\ref{Sperm}, many lemmas here are local; only
Lemmas~\ref{lem_uniendp_sum},~\ref{lem_cv_paths} and~\ref{LpM} are used
outside this section. 
\subsection{Generating-functionology}
We start by introducing some generating functions.
First, the generating function of
rooted plane trees enumerated by edges:
\begin{equation}
B(z):=\frac{1-\sqrt{1-4z}}{2z}
=\sumno \Cat{n}z^n,
\end{equation}
where $\Cat{n}:=\frac{1}{n+1}\binom{2n}{n}$ is the $n$th Catalan number.
We also introduce
\begin{equation}
T(z):=zB(z)=\frac{1-\sqrt{1-4z}}{2},
\end{equation}
which satisfies 
\begin{equation}\label{T=}
T(z)=\frac{z}{1-T(z)}.
\end{equation}
We will also use doubly rooted plane trees. These trees have two roots,
labelled first and second root. Both roots are corners of $T$; the roots may
be the same corner, but that case we also distinguish between two different 
orderings of the roots. 
A rooted plane tree with $n$ edges has $2n$ corners, and thus a second root
may be added in $2n+1$ different places (including 2 places in the corner of
the first root).
Therefore, the generating function of doubly rooted plane trees, enumerated
by edges, 
is
\begin{equation}\label{AB}
A(z):=\Bigpar{2z\frac{\partial}{\partial z}+1}B(z)
.\end{equation}

We recall the Lagrange--Bürmann formula that will be useful to us in this
section. %, see \eg{} \cite[p.~732]{FS}.
\begin{theorem}[Lagrange--Bürmann formula (\cite{FlaSeg09}, Theorem A.2)]
Let $F$ and $\phi$ be power series satisfying 
\begin{equation}\label{LB}
F(z)=z\phi(F(z)),
\end{equation}
then, for any (analytic) function $f$, we have
\begin{equation}\label{Lagrange}
[z^n]f(F(z))=\frac 1 n [z^{n-1}]\bigpar{\phi(z)^nf'(z)},
\qquad n\ge1.
\end{equation}
\end{theorem}

We will use this formula to estimate the number of occurrences of a certain pattern in $\T_n$.

\begin{lemma}\label{LT1}
Let $\bt$ be a rooted plane tree of size $\ell$.
Then the number of rooted plane trees of size $n$ with a marked rooted subtree 
isomorphic to
$\bt$ 
is 
\begin{equation}\label{lt1a}
T_{n,\ell}=2n[z^{n-\ell}]B(z)^{2\ell}
=2\ell\binom{2n}{n-\ell}
,
\qquad n\ge\ell,
\end{equation}
and we have
%Hence, 
%the expected number of subtrees of\/ $\T_n$ that are isomorphic to $\bt$ is
\begin{equation}\label{eq_Tnl_equiv}
\E N_{\bt}(\T_n)=
\frac{T_{n,\ell}}{\Cat{n}}=(1+o(1))2\ell n
\end{equation}
where the $o(1)$ is uniform over $\ell\in[1,\Linf]$.
\end{lemma}

\begin{proof}
We want to enumerate the number of trees of size $n$ with a marked subtree
isomorphic to $\bt$. One can build such a tree in a bijective way (see
Figure~\ref{fig_Tnl}). 
Starting from a copy of $\bt$, to each corner $c$ of $\bt$, pick a rooted
tree and glue its root corner to $c$. The only constraint is that the total
size of all the $2\ell$ trees that we graft must be $n-\ell$. One obtains an
unrooted 
tree $T$ of size $n$ with a marked copy of $\bt$. 
Finally, one just needs to pick one of its $2n$ corners as the root;
note that the resulting pairs $(T,\bt)$  will be distinct, since $T$ has no
automorphisms (preserving order and $\bt$).
Hence the number of rooted such trees is 
\begin{equation}\label{emT}
T_{n,\ell}=2n[z^{n-\ell}]B(z)^{2\ell}.
\end{equation}

\begin{figure}
\center
\includegraphics[scale=1]{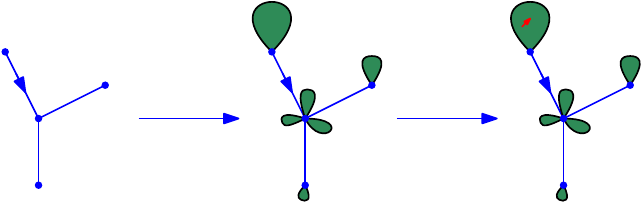}
\caption{Building a tree $T$ with a marked copy of $\bt$. Here $\ell=3$, $\bt$ is in blue, the $2\ell$ trees are in green, and the root of $T$ is in red.}\label{fig_Tnl}
\end{figure}

%We now turn to the proof of~\eqref{eq_Tnl_equiv}. 
Using the Lagrange--Bürmann formula \eqref{Lagrange} and \eqref{T=},
we get 
\begin{align}
[z^{n-\ell}]B(z)^{2\ell}&=[z^{n+\ell}]T(z)^{2\ell}
\notag\\
&=\frac{1}{n+\ell}[z^{n+\ell-1}]\Bigpar{\frac{1}{(1-z)^{n+\ell}}2\ell z^{2\ell-1}}
\notag\\
&=\frac{2\ell}{n+\ell}[z^{n-\ell}]\frac{1}{(1-z)^{n+\ell}}
\notag\\
&=\frac{2\ell}{n+\ell}\binom{2n-1}{n-\ell}
=\frac{2\ell}{2n}\binom{2n}{n-\ell}
.\end{align}
This, together with \eqref{emT}, shows
\eqref{lt1a}.

The Stirling formula gives us 
\begin{equation}\label{stirling2}
\binom{2n}{n-\ell}\sim \binom{2n}{n}=(n+1)\Cat{n}
\end{equation}
uniformly for $|\ell|\le \Linf$ (because $\Linf=o(n\qq)$); hence
\eqref{lt1a} implies
\begin{equation}
\frac{T_{n,\ell}}{\Cat{n}}\sim 2\ell n
\end{equation}
uniformly in $\ell\in[1,\Linf]$, which shows \eqref{eq_Tnl_equiv}.
\end{proof}

With the same method, we can also get an estimate for the number of pairs of patterns.
\begin{lemma}\label{lem_non_intersecting_trees}
Let $\bt_1$ and $\bt_2$ be two rooted plane trees
of sizes $\ell_1$ and $\ell_2$, 
and let
$\ell=\ell_1+\ell_2$.
Then the number of rooted plane trees of size $n$ with a marked pair of non intersecting rooted
subtrees that are isomorphic to $\bt_1$ and $\bt_2$, respectively, 
is
\begin{equation}\label{l2ta}
T_{n,\ell_1,\ell_2}\leq 8n\ell_1\ell_2[z^{n-\ell}]\bigpar{A(z)B(z)^{2\ell-2}}
=4\ell_1\ell_2(n+\ell)\binom{2n}{n+\ell}.
\end{equation}
and we have
\begin{equation}\label{eq_Tnl1l2_equiv}
\frac{T_{n,\ell_1,\ell_2}}{\Cat{n}}\leq (1+o(1))(2\ell_1 n)(2\ell_2 n)
\end{equation}
where the $o(1)$ is uniform over $\ell_1,\ell_2\in[1,\Linf]$.
\end{lemma}

\begin{remark}
It can be shown that the inequality in~\eqref{eq_Tnl1l2_equiv} is actually
an equality, but we do not need this here.
\end{remark}

\begin{proof}
This is similar to the proof of \refL{LT1}.
The decomposition now is the following (see Figure~\ref{fig_Tnl1l2}): start from
a copy of $\bt_1$ and a copy of $\bt_2$, choose one corner on each and graft
a doubly rooted tree $\tilde T$, identifying its first root (second root) with
the root corner of $\bt_1$ ($\bt_2$). Then graft rooted trees to
each of the $2\ell-2$ remaining corners of $\bt_1$ and $\bt_2$ (as done in
the proof of \refL{LT1}), to obtain an unrooted tree
$T$ of size $n$, and pick one of its $2n$ corners as
the root. This way, we can build all rooted trees with non intersecting copies of $\bt_1$ and $\bt_2$, plus some cases where they intersect (namely, when $\bt_1$ and $\bt_2$ are grafted at a same vertex of $\tilde T$).

\begin{figure}
\center
\includegraphics[scale=0.8]{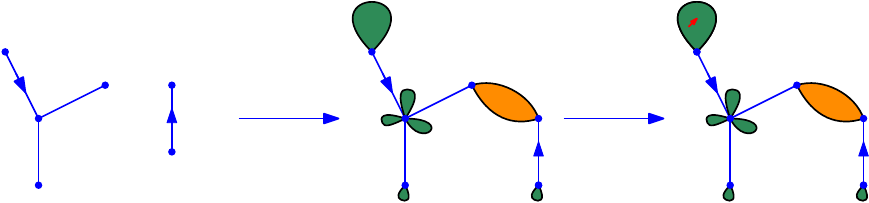}
\caption{Building a tree $T$ with a marked pair of trees. Here $\ell_1=3$ and $\ell_2=1$, $\bt_1$ and $\bt_2$ are in blue, the doubly rooted tree is in orange, the $2\ell-2$ trees are in green, and the root of $T$ is in red.}\label{fig_Tnl1l2}
\end{figure}

This yields
\begin{align}
T_{n,\ell_1,\ell_2}\leq (2\ell_1)(2\ell_2)[z^{n-\ell}]\bigpar{A(z)B(z)^{2\ell-2}}\cdot(2n),
\end{align}
showing the first part of \eqref{l2ta}.
We then use \eqref{AB} and, again, 
the Lagrange--Bürmann formula, and obtain
\begin{align}
[z^{n-\ell}]\bigpar{A(z)B(z)^{2\ell-2}}&
=2[z^{n-\ell-1}]\bigpar{B'(z)B(z)^{2\ell-2}}+[z^{n-\ell}]B(z)^{2\ell-1}
\notag\\
&=2[z^{n-\ell-1}]\frac{(B^{2\ell-1})'}{2\ell-1}+[z^{n-\ell}]B(z)^{2\ell-1}
\notag\\
&=\left(2\frac{n-\ell}{2\ell-1}+1\right)[z^{n-\ell}]B(z)^{2\ell-1}
\notag\\
&=\left(2\frac{n-\ell}{2\ell-1}+1\right)\frac{2\ell-1}{n+\ell-1}\binom{2n-2}{n-\ell}
\notag\\
&=\frac{2n-1}{n+\ell-1}\binom{2n-2}{n-\ell}
=\frac{n+\ell}{2n}\binom{2n}{n-\ell}
,\end{align}
and \eqref{l2ta} follows.
Using \eqref{stirling2} again, we obtain
\begin{equation} 
\frac{T_{n,\ell_1,\ell_2}}{\Cat{n}}\leq (1+o(1)) (2n\ell_1)(2n\ell_2)
,\end{equation}
uniformly for $\ell_1,\ell_2\le\Linf$.
\end{proof}

\subsection{Unions of paths}
\begin{definition}
The set $\Bicx(\ell_1,\ell_2)$ is the set of unrooted
trees $\bt$ with edges either blue, red, or bicolored such that 
$\bt$ is the union of a
blue path of length $\ell_1$ and a red path of length $\ell_2$. 
\end{definition}
\begin{lemma}\label{lem_bic_cardinal}
There are at most $16(\ell_1+1)(\ell_2+1)(\min(\ell_1,\ell_2)+1)$ trees in
$\Bicx(\ell_1,\ell_2)$.  
\end{lemma}

\begin{proof}
In this proof, we make an exception, and allow paths to have length 0.
We describe a procedure to build a tree in $\Bicx(\ell_1,\ell_2)$ 
(see Figure~\ref{fig_path_union}).
\begin{enumerate}
\item Create a bicolored path $p'$ of length $0\leq \ell'\leq \min(\ell_1,\ell_2)$ ($\min(\ell_1,\ell_2)+1$ possibilities).
\item Create two red paths $p_1^a$ and $p_1^b$ of total length
  $\ell_1-\ell'$ ($\leq \ell_1+1$ possibilities).
\item Create two blue paths $p_2^a$ and $p_2^b$ of total length $\ell_2-\ell'$ ($\leq \ell_2+1$ possibilities).
\item Attach $p_1^a$ and $p_2^a$ to $\start(p')$ ($2$ possibilities).
\item Attach $p_1^b$ and $p_2^b$ to $\eend(p')$ ($2$ possibilities).
\item Orient the blue and red paths ($2\times 2=4$ possibilities).
\end{enumerate}
This procedure is surjective; hence this proves what we wanted.
\end{proof}

\begin{figure}
\center
\includegraphics[scale=0.5]{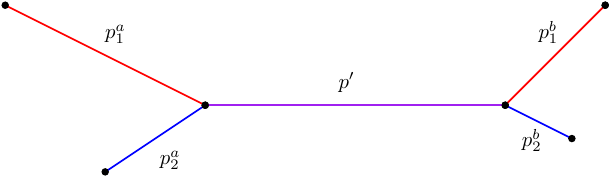}
\caption{The union of two paths}\label{fig_path_union}
\end{figure}

Now, for $i,j\leq \Mmax$, we define 
\begin{align}
  \label{Bic}
\Bic(i,j)=\bigsqcup \Bicx(\ell_1,\ell_2),
\end{align}
where the union is over all $\ell_1,\ell_2\in \Interv{i}\times\Interv{j}$.

\begin{lemma}
For every tree $T$ and $\bm\in\setM$, we have
\begin{equation}\label{eq_bound_disjoint_paths}
1\geq \frac{\pM(T)}{\prod_{i=1}^{s(\m)}P_{m_i}(T)}\geq 1-\sum_{i,j}\frac{\sum_{\bt\in\Bic(m_i,m_j)}N_\bt(T)}{P_{m_i}(T)P_{m_j}(T)}
\end{equation}
\end{lemma}

\begin{proof}
The proof is direct from the inclusion--exclusion principle: indeed,
$\prod_{i=1}^{s(\m)}P_{m_i}(T)$ counts lists of $s(\m)$ paths with the
right sizes, without any constraint of non-intersection between these
paths. Additionally, 
\begin{align} 
\sum_{i,j}\sum_{\bt\in\Bic(m_i,m_j)}N_\bt(T)\prod_{k=1\atop k\neq i,j}^{s(\m)}P_{m_k}(T)
\end{align}
(over)counts such lists where two of the paths intersect.
\end{proof}

\begin{definition}\label{DPU}
A \emph{path tree} is a tree $T$ together with a list of $q\ge1$ paths
$p_1,p_2,\dots ,p_q$ 
such that $T=\bigcup_{i=1}^q p_i$, and  for every  $i>1$, there
exists $j<i$ such that $\Ext(p_i)\cap\Ext(p_j)\neq \emptyset$.
(For convenience, we denote the path tree simply by $T$.) 
For a path tree $T$, we write $\Ext(T):=\bigcup_{i=1}^q\Ext(p_i)$.

Let $\PU_{q,w}(\ell)$ be the set of path trees $T$ 
with $q$ paths $p_i$ such that $\bigabs{\Ext(T)}=w$,
and  $|p_i|\le\ell$ for every path $p_i$.
\end{definition}
%Note that we must have $w\le 2q$; otherwise $\PU_{q,w}(\ell)$

\begin{lemma}\label{lem_union_endpoints}
  \begin{romenumerate}
  
  \item\label{i2} 
If $T\in \PU_{q,w}(\ell)$, then $\max_{v\in V(T)} \deg(v)\leq q+1$.
  \item \label{i3}
For every $q$ and $w$, there exists a constant $C_{q,w}$ such that
$|\PU_{q,w}(\ell)|\leq C_{q,w}\ell^{2w-3}$.
\end{romenumerate}
\end{lemma}

\begin{proof}
We will prove this by induction. 
Both parts are verified for $q=1$ because  $\PU_{1,w}$ is empty unless
$w=2$, and
$\PU_{1,2}(\ell)$ is just the set of paths of length $\leq \ell$,
so $|\PU_{1,2}(\ell)|=\ell$.

Now assume $q\ge2$, and
let $T=\bigcup_{i=1}^q p_i\in \PU_{q,w}(\ell)$.
Then $T':=\bigcup_{i=1}^{q-1} p_i$ is also a path tree,
with $T'\in\PU_{q-1,w'}(\ell)$ for some $w'\in\set{w-1,w}$.

Starting from $T'$, one can reconstruct $T$ by adding a path $p_q$.
If $w'=w$, then both endpoints of $p_q$ have to be in $\Ext(T')$,
which yields $\le w^2$ %\le(2q)^2$ 
choices.

If $w'=w-1$, then $p_q$ must have one endpoint $v$ 
in $\Ext(T')$, but its other
endpoint may be either in $T'\setminus\Ext(T')$ or outside $T'$; in the
latter case, let $v'$ by last point in $p_q\cup T'$ (starting from $v$).
We may then reconstruct $T$ from $T'$ as follows (with some overcounting,
since not all choices below are allowed):
\begin{enumerate}
\item Choose a vertex $v\in \Ext(T')$ ($w'\le w$ possibilities).
\item Choose a vertex $v'\in V(T')$ ($\leq |V(T')|\leq
  q\ell$ possibilities).
\item Either stop at $v'$ and let $v^*:=v'$, 
or attach a path of length $\le \ell $ 
to one of the corners of $v'$, and let $v^*$ be the other endpoint of this
path ($\leq \ell q$ possibilities, by \ref{i2}).
\item Declare the path from $v$ to $v^*$ as $p_q$, and give it an
  orientation (2 choices).
%\item if the resulting tree has size $\leq \ell$, return it, otherwise return nothing.
\end{enumerate}

It is clear that with this procedure, the vertex degrees increase by at most
$1$ and thus \ref{i2} holds for $T$, since it holds for
$T'$ by  the induction hypothesis.

Furthermore, since the procedure is surjective, 
it follows that
\begin{align}
  |\PU_{q,w}(\ell)|\le w^2|\PU_{q-1,w}(\ell)|
+2wq\ell\bigpar{1+q\ell}|\PU_{q-1,w-1}(\ell)|,
\end{align}
and \ref{i3} follows by induction.
\end{proof}

\begin{figure}
\center
\includegraphics[scale=1]{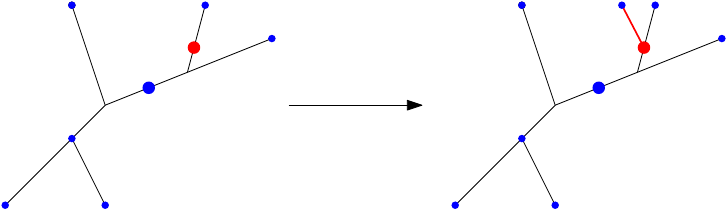}
\caption{Building a path tree recursively. Here, $w'=7$ and $w=8$. The vertices of
  $\Ext(T')$ are blue, $v$ is the large blue dot, and $v'$ is the large red dot. The path between $v'$ and $v^*$ is in red.}\label{fig_path_tree}
\end{figure}

\subsection{Patterns in uniform trees}

\begin{lemma}\label{lem_bic_sum}
Fix $i,j\ge0$. Then
\begin{equation}
\E\lrsqpar{\sum_{\bt\in\Bic(i,j)}N_\bt(\T)}
=O\bigpar{n\Lmaxx^6}
=o\bigpar{n^2\Lmaxx^4}.
\end{equation}
\end{lemma}

\begin{proof}
A tree $\bt$ in $\Bic(i,j)$ will have size at most $\frac{i+j+2}{\Mmax}\Lmax
= O(\Lmax)$.
Hence, by \eqref{Bic} and
Lemma~\ref{lem_bic_cardinal}, the cardinality of $\Bic(i,j)$ is bounded
by $\left(\frac{\Lmax}{\Mmax}+1\right)^2\times O\bigpar{\Lmaxx^3}
=O\bigpar{\Lmaxx^5}$. 
Since $\Lmax=o(\Linf)$, 
we may also use~\eqref{eq_Tnl_equiv}.
Consequently, 
\begin{align}
\E\lrsqpar{\sum_{\bt\in\Bic(i,j)}N_\bt(\T)}
&=\sum_{\bt\in\Bic(i,j)}\E N_\bt(\T) \notag\\
%&=\sum_{\bt\in\Bic(i,j)} \frac{T_{n,|\bt|}}{\Cat{n}}\notag\\
&=\sum_{\bt\in\Bic(i,j)} (1+o(1)) 2n|\bt|\notag\\
%&\leq (1+o(1))\frac{i+j+2}{\Mmax}\Lmax |\Bic(i,j)|\notag\\
&= |\Bic(i,j)|\cdot O(n\Lmax)\notag\\
&=O(n\Lmaxx^6),
\end{align}
and we conclude by recalling that $\Lmax=o(\sqrt n)$.
\end{proof}
Here is a similar lemma that uses Lemma \ref{lem_union_endpoints}
and the notation %$\PU_{q,w}(\ell)$ 
there.
\begin{lemma}\label{lem_uniendp_sum}
For every fixed $q$, $w$, and $c$, we have 
\begin{equation}\label{alla}
\E\lrsqpar{\sum_{\bt\in\PU_{q,w}(c\Lmax)}N_\bt(\T)}
=o\bigpar{ n\Lmaxx^{2w-2} \log g}.
\end{equation}
\end{lemma}
\begin{proof}
Since $\bt\in\PU_{q,w}(c\Lmax)$ implies $|\bt|\le qc\Lmax=o(\Linf)$, 
we have, %(for large $n$)
by \eqref{eq_Tnl_equiv} and Lemma~\ref{lem_union_endpoints},
\begin{align}
\E\lrsqpar{\sum_{\bt\in\PU_{q,w}(c\Lmax)}N_\bt(\T)}
%&=\sum_{\bt\in\PU_{q,w}(c\Lmax)} \frac{T_{n,|\bt|}}{\Cat{n}}\notag\\
&=\sum_{\bt\in\PU_{q,w}(c\Lmax)} \E N_\bt(\T)\notag\\
&=\sum_{\bt\in\PU_{q,w}(c\Lmax)}(1+o(1))2n|\bt|\notag\\
&\leq (1+o(1))|\PU_{q,w}(c\Lmax)|\, 2n qc\Lmax\notag\\
&=O(n\Lmaxx^{2w-2})
,\end{align}
which implies \eqref{alla}.
\end{proof}

\begin{lemma}\label{lem_cv_paths}
For every fixed  $i\ge0$,
\begin{equation}\label{allb}
\frac{P_{i}(\T)}{n\Lmaxx^2}\pto \frac{2i+1}{\Mmax^2}.
\end{equation}
\end{lemma}

\begin{proof}
We will prove a slightly stronger result, namely $L^2$ convergence.

Let $J_i:=\Interv{i}$.
Using~\eqref{Pit} and \eqref{eq_Tnl_equiv}, we have 
\begin{equation}\label{allc}
\E\bigsqpar{P_{i}(\T)}
%\sum_{\ell=\lceil\frac{i}{\Mmax}\Lmax\rceil}^{\ceil{\frac{i+1}{\Mmax}\Lmax}-1} 
=\sum_{\ell\in J_i}\frac{T_{n,\ell}}{\Cat{n}}
=\bigpar{1+o(1)}\sum_{\ell\in J_i}2\ell n
\sim n\Lmaxx^2\frac{2i+1}{\Mmax^2}.
\end{equation}

Now, let us count pairs of paths in $\T$, in order to estimate 
$\E\bigsqpar{P_i(\T)^2}$. There are two cases:
\begin{enumerate}
\item 
Two disjoint paths. Such pairs are enumerated by
Lemma~\ref{lem_non_intersecting_trees}, 
which implies that
the total expectation of the number of such pairs is
\begin{align}
 \sum_{\ell_1,\ell_2\in J_i} \frac{T_{n,\ell_1,\ell_2}}{\Cat{n}}
\le \bigpar{1+o(1)}\sum_{\ell_1,\ell_2\in J_i}(2\ell_1 n)(2\ell_2 n),
\end{align}
which by comparison with \eqref{allc} is $\;\sim \bigpar{\E P_i(\T)}^2$.

\item Two paths that intersect in at least one vertex. 
Their union is then a bicolored tree, and by Lemma~\ref{lem_bic_sum}, the
expectation of the number of such pairs is $o(n^2\Lmaxx^4)$,
which by \eqref{allc} is $o \bigpar{\E P_i(\T)}^2$.
\end{enumerate}
In summary, this establishes that
\begin{equation}
\E\bigsqpar{P_i(\T)^2}\sim\bigpar{ \E P_i(\T)}^2,
\end{equation}
which implies $P_i(\T)/\E P_i(\T)\to1$ in $L^2$.
By \eqref{allc}, this yields \eqref{allb} in  $L^2$, and thus in probability.
\end{proof}

\begin{lemma}\label{LpM}
For every fixed  $\m \in\setM$,
\begin{equation}
\frac{\pM(\T)}{ \prod_{i=1}^{s(\m)}P_{m_i}(\T)}
\pto1.
\end{equation}
\end{lemma}
\begin{proof}
For fixed $i$ and $j$, we have by Lemma~\ref{lem_bic_sum} and the Markov
inequality 
\begin{equation}
\frac{1}{n^2\Lmaxx^4}\sum_{\bt\in\Bic(i,j)}N_\bt(\T)\pto0.
\end{equation}
Hence by Lemma~\ref{lem_cv_paths}, for every fixed $i$ and $j$,
\begin{equation}
\frac{\sum_{\bt\in\Bic(i,j)}N_\bt(\T)}{P_i(\T)P_j(\T)}\pto0,
\end{equation}
and we conclude by using the inequalities
\eqref{eq_bound_disjoint_paths}. 
\end{proof}

\section{Cycles in \ctree{s}}\label{Sctrees}

\subsection{Definitions}
If $(T,\sigma)$ is a \ctree, then any simple cycle of length $\ell$ in its
underlying graph can be decomposed into a list $\bP=(p_1,p_2,\ldots,p_k)$ 
of non-intersecting simple paths in $T$ 
such that 
\begin{PXenumerate}{C}
\item\label{C1} $\sum_{i=1}^k |p_i|=\ell$;
\item\label{C2} $\eend(p_i)\sim \start(p_{(i+1 \mod k)})$ for all $i$;
\item \label{C3} 
for every other pair of vertices $v,v'\in (p_1,p_2,\ldots,p_k)$, we
  have $v\not\sim v'$. 
\end{PXenumerate}
This decomposition is unique up to cyclically reordering the $p_i$, or
reversing them all and their order, or a combination of both.
Conversely, every list satisfying \ref{C1}--\ref{C3} yields a simple cycle
in the underlying graph.

For two lists $\bP,\bP'\in\Poo(T)$, 
we write $\bP\simeqx \bP'$ if and only if $\bP'$ can be obtained from 
$\bP$ by cyclically reordering its paths, or
reversing them all and their order, or a combination of both.
Note that $\bP\simeqx \bP'$ entails $s(\bP)=s(\bP')$ and
$\ell(\bP)=\ell(\bP')$, and that each list $\bP$ is in an equivalence class
$[\bP]$ with exactly $2s(\bP)$ elements.
Let $\hPoo(T)$ be a subset of $\Poo(T)$ obtained by selecting exactly one
element from each equivalence class in $\Poo(T)$.

Given also a \cperm{} $\gs$ of the vertex set of $T$, 
so that $(T,\gs)$ is a \ctree, 
let $\CC(T,\gs)$ be the set of lists $\bP\in\hPoo(T)$
that satisfy \ref{C2}--\ref{C3} above. 
There is thus a 1--1 correspondence between $\CC(T,\gs)$ and the set of
simple cycles in the underlying graph.
Let also, recalling \eqref{PPkm}, 
\begin{align}
\CCkm(T,\gs)&:=\CC(T,\gs)\cap\PPkm(T),\label{CCkm}
\\\label{CCkab}
\CCkab(T,\gs)&:=
%\CC(T,\gs)\cap\PPkab(T)=
\bigcup_{a\leq m< b} \CCkm(T,\gs),
\end{align}
and denote the cardinalities of these sets by
$\cckm(T,\gs):=|\CCkm(T,\gs)|$
and
$\cckab(T,\gs):=|\CCkab(T,\gs)|$.

Furthermore,
let $\tCCkm(T,\gs)$ be the set of lists $\bP\in\PPkm(T)\cap\hPoo(T)$
that satisfy \ref{C2}.
Thus $\tCCkm(T,\gs)\supseteq\CCkm(T,\gs)$.
Let further 
$\tcckm(T,\gs):=|\tCCkm(T,\gs)|$,
and note that $\tcckm(T,\gs)\ge\cckm(T,\gs)$.

We ultimately want to work with random trees, but it will be easier to work
with deterministic sequences of trees at first. We say that a sequence
$(T_n)$ is a \emph{good sequence of trees} if for all $n$, $T_n$ is a tree
of size $n$ and that the  following properties hold:
for every fixed $M\ge1$
and every fixed $\m\in \setM$,
\begin{equation}\label{good_paths}
\pM(T_n)\sim \left(\frac{n\Lmaxx^2}{\Mmax^2}\right)^{s(\m)}\Pm
=\left(\frac{n^2}{12\Mmax^2 g}\right)^{s(\m)}\Pm
,\end{equation}
and for each fixed $q$, $w$, and $c$,
with $\PU_{q,w}(\ell)$ defined in \refD{DPU},
\begin{equation}\label{good_unionpaths}
\sum_{\bt\in\PU_{q,w}(c\Lmax)}N_\bt(T_n) = O\bigpar{ n \Lmaxx^{2w-2}\log g}.
\end{equation}

We will see in \refL{Lgood} that we may assume that the sequence $\T_n$ of
random trees is good. % (almost surely).

\subsection{Expectation}

\begin{lemma}\label{Lexpectation}
Let $(T_n)$ be a good sequence of trees. Then, for every $M\ge1$, $m\ge0$
and $k\ge1$, 
as \ntoo,
\begin{equation}\label{lc2a}
\E\bigsqpar{\cckm(T_n,\sig)}\to 
\sum_{|\bm|=m,s(\bm)=k}
\frac{\Pm}{2k}\left(\frac{1}{2\Mmax^2}\right)^{k}
=:\gL_k(m).
\end{equation}
Furthermore,
\begin{equation}\label{lc2b}
\E\bigsqpar{\tcckm(T_n,\sig)-\cckm(T_n,\sig)}\to 0.
\end{equation}
\end{lemma}

\begin{proof}
We start by estimating $\E\tcckm(T_n,\sig)$. 
For each given list $\bP=(p_1,\dots,p_k)\in\hPoo(T_n)$,
let $\tpi(\bP)$ be the probability  that \ref{C2} holds. % for  $(T_n,\bgs)$.
Then, by definitions and symmetry,
\begin{align}\label{lc2c}
\E\bigsqpar{\tcckm(T_n,\sig)}
= \sum_{\bP\in\PPkm(T_n)\cap\hPoo(T_n)}\tpi(\bP)
=\frac{1}{2k}\sum_{\bP\in\PPkm(T_n)}{\tpi(\bP)}.
\end{align}

To find $\tpi(\bP)$, note that
we may relabel the $2k$ endpoints in $\Ext(\bP)$ as $1,\dots,2k$ 
in an order such  that \ref{C2} becomes $2i-1\sim 2i$ for $i=1,\dots,k$,
\ie, that $2i-1$ and $2i$ belong to the same cycle in $\bgs$.
There are two cases: either these $k$ cycles are distinct, or at least two
of them coincide.
The first event is 
$\cE_{n,g}^{(k)}$ in \refL{LXC}, and that lemma shows that
its probability is
\begin{equation}\label{ERII}
\P_{n,g}^{(k)}=
\P\bigpar{\cE_{n,g}^{(k)}}
\sim \left(\frac{6g}{n^2}\right)^k.
\end{equation}
The second event means that $1,\dots,2k$ belong to at most $k-1$ different
cycles of $\bgs$. Hence, \refL{LXD} shows that the
probability of this event is
\begin{align}\label{lc2d}
  \sum_{r=1}^{k-1}O\left(\frac{g^r}{n^{2k}}\right)
=o\Bigpar{\left(\frac{g}{n^2}\right)^{k}}.
\end{align}
Summing \eqref{ERII} and \eqref{lc2d}, we see that 
\begin{align}\label{pi1}
  \tpi(\bP)\sim \left(\frac{6g}{n^2}\right)^k
%= \left(\frac{6g}{n^2}\right)^{s(\bP)}
,\end{align}
uniformly for all $\bP$ with  $s(\bP)=k$.

We develop the sum in \eqref{lc2c} 
using \eqref{pi1} and \eqref{good_paths}, %\eqref{eq_Pm}, 
noting that if $\bP\in\PPm$, then $s(\bP)=s(\bm)$;
we thus obtain
\begin{align}\label{pia}
  \E\bigsqpar{\tcckm(T_n,\bgs)} &
= \frac{1}{2k}\summmk \pM(T_n)
(1+o(1))\left(\frac{6g}{n^2}\right)^{k}
\notag\\&
= \frac{1}{2k}\left(\frac{1}{2\Mmax^2}\right)^{k}\summmk\Pm+o(1).
\end{align}

Next, we consider the difference $\tcckm(T_n,\bgs)-\cckm(T_n,\bgs) $.
A given list $\bP=(p_1,\dots,p_k)\in\hPoo(T_n)$ belongs to 
$\tCCkm(T_n,\bgs)\setminus \CCkm(T_n,\bgs)$ if
it satisfies \ref{C2} but not \ref{C3}. 

This means that  one of the following holds:
\begin{itemize}
\item The $2k$ endpoints belong to at most $k-1$ cycles, which
  happens with probability $o\bigpar{\left(\xfrac {g }{n^2}\right)^k}$
  by~\eqref{lc2d}.
\item There exists a vertex $v$ in $\bP\setminus\Ext(\bP)$ such that $v$
  belongs to the same cycle as $\eend(p_i)$ and $\start(p_{i+1})$ for some
  $i$. Hence, we have $2k+1$ points belonging to $k$ cycles, which by
  Lemma~\ref{LXD}, happens with probability $O\left(\xfrac
    {g^k}{n^{2k+1}}\right)$ for a given $v$. There are 
$O(\Lmax)=O\left(\sqrt{\xfrac n g }\right)$ vertices in $\bP$; 
hence by a union bound, the probability
  of this event is  
\begin{equation}
O\left(\frac {g^k}{n^{2k+1}}\right)\times O\left(\sqrt{\frac n g }\right)=o\left(\left(\frac g {n^2}\right)^k\right).
\end{equation}
\item There exist two vertices $v,v'$ in $\bP\setminus\Ext(\bP)$ such that
  $v$ and $v'$ belongs to the same cycle. Hence, we have $2k+2$ points
  belonging to $k+1$ cycles, which by Lemma~\ref{LXD}, happens with
  probability $O\left(\xfrac {g^{k+1}}{n^{2k+2}}\right)$ for a given pair
  $v,v'$. There are $O\bigpar{\Lmaxx^2}=O\left(\xfrac n g \right)$ pairs of
  vertices in $\bP$; 
  hence by a union bound, the probability of this event is  
\begin{equation}
O\left(\frac {g^{k+1}}{n^{2k+2}}\right)\times O\left(\frac n g \right)
=o\left(\left(\frac g {n^2}\right)^k\right).
\end{equation}
\end{itemize}

Hence, letting $\pi(\bP)$ be the probability that both \ref{C2} and
\ref{C3} hold, we have
\begin{align}
  0\le \tpi(\bP)-\pi(\bP) 
=o\left(\left(\frac g {n^2}\right)^k\right),
\end{align}
uniformly for all $\bP$ with  $s(\bP)=k$.
Consequently,
arguing as in \eqref{pia}, but more crudely, using
again \eqref{good_paths},
\begin{align}\label{pib}
  \E\bigsqpar{\tcckm(T_n,\bgs) -\cckm(T_n,\bgs)} &
=\summmk\frac{ \pM(T_n)}{2k}
o\Bigpar{\Bigpar{\frac{g}{n^2}}^{k}}
%\notag\\&
=o(1).
\end{align}
The proof is completed by \eqref{pia} and \eqref{pib}.
\end{proof}

\subsection{Higher moments}
\begin{lemma}\label{Lmom}
Let $(T_n)$ be a good sequence of trees, 
let $(m_1,k_1),\dots,(m_q,k_q)$ be distinct pairs of integers in $\bbZgeo\times\bbZ_+$,
and let $r_1,r_2,\dots ,r_q$ be fixed positive integers, for some $q\ge1$.
Then
\begin{equation}\label{lc3}
\E\Bigsqpar{\prod_{i=1}^{q}\left(\tccx{m_i}_{k_i}(T_n,\bgs)\right)_{r_i}}
= \prod_{i=1}^{q}\left(\E\bigsqpar{\tccx{m_i}_{k_i}(T_n,\bgs)}\right)^{r_i}+o(1)
= \prod_{i=1}^{q}\gLxx{k_i}{m_i}^{r_i}+o(1)
.\end{equation}
\end{lemma}

The proof is somewhat lengthy, but the main idea is to show that, given a fixed
number of distinct cycles in $(T_n,\sig)$ of lengths $O(\Lmax)$, they are
pairwise disjoint whp.  

\begin{proof}

We argue similarly as in the special case $q=1$ and $r_1=1$  in
\refL{Lexpectation},
and write this expectation as
\begin{align}\label{ab}
  \E%\Bigsqpar
{\prod_{i=1}^{q}\left(\tcckmi(T_n,\bgs)\right)_{r_i}}
=\hsumx_{(\bP(i,j))_{ij}} \tpi((\bP(i,j))_{ij}),
\end{align}
where we sum over all sequences of distinct lists 
%$(\bP(i,j))_{\substack{1\le i\le  q\\1\le j\le r_i}}$ 
$(\bP(i,j))_{1\le i\le  q,\ 1\le j\le r_i}$ 
such that 
$\bP(i,j)\in \PPkmi(T_n)\cap \hPoo(T_n)$, 
and $\tpi((\bP(i,j))_{ij})$ is the probability that every
$\bP(i,j)\in\tCCkmi(T_n,\bgs)$. 
Recalling the definition of $\hPoo$, 
we can rewrite \eqref{ab} as
\begin{align}\label{ac}
  \E \prod_{i=1}^q \bigpar{\tcckmi(T_n,\bgs)}_{r_i} 
= \sumx_{(\bP(i,j))_{ij}}\frac{ \tpi((\bP(i,j))_{ij})}{\prod_{i,j}{(2k_i)}},
\end{align}
where we now sum over all sequences of  lists 
$(\bP(i,j))_{ij}$ such that 
$\bP(i,j)\in \PPkmi(T_n)$ and no two $\bP(i,j)$ are equivalent (for
$\simeqx$). 

%First, we show that the $\Ext(\bP(i,j))$ are pairwise disjoint whp. 

For each such sequence $(\bP(i,j))_{ij}$, define a graph $H$ with vertex
set $V(H):=\bigcup_{i,j}\Ext(\bP(i,j))$, the set of endpoints of all
participating paths, and edges of two colours as follows:
For each list $\bP(i,j)=(p_\mnu)_1^k$, add for each $\mnu$ a \emph{green} edge
between $\start(p_\mnu)$ and $\eend(p_\mnu)$, 
and a \emph{red} edge between $\eend(p_\mnu)$ and $\start(p_{\mnu+1})$ (see Figure~\ref{fig_path_graph} for an example). 
(We use here and below the convention $p_{k+1}:=p_1$.) 
Hence, by the definitions above,
every $\bP(i,j)\in\tCCkmi(T_n,\bgs)$
if and only if 
each red edge in $H$ joins two vertices in the same cycle of $\bgs$.
Thus, $\tpi((\bP(i,j))_{ij})$ in \eqref{ac} is the probability of this event.

\begin{figure}
\center
\includegraphics[scale=0.8]{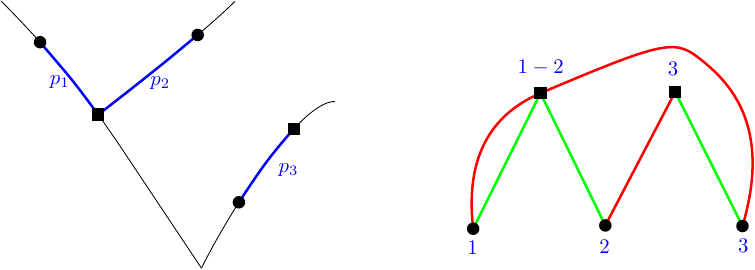}
\caption{Three paths in a tree (left) and their associated graph $H$
  (right). The paths are in blue, the start of a path is represented as a
  square, and its end as a dot.}\label{fig_path_graph} 
\end{figure}

For each  graph $H$ constructed in this way, let $H_G$ %[$H_R$]
be the subgraph consisting of all green %[red] 
edges, and say that a connected component of $H_G$ is a \emph{green
  component} of $H$. 
Define a \emph{red component} in the same way, and let $\gq_G(H)$ and $\gq_R(H)$ be
the numbers of green and red components, respectively.

Let $M_H=M_H(n)$ be the number of terms in \eqref{ac} with a given graph $H$.
(For some fixed $q$, $m_1,\dots,m_q$, $k_1,\dots,k_q$, and $r_1,\dots,r_q$.)
We estimate $M_H$ as follows.
Each green component of $H$ with $v$ vertices corresponds to some 
set of paths $p_{i,j,\mnu}$ such that
their union is a connected subtree $\bt$ of $T_n$. All these paths have
lengths  $O(\Lmax)$.
Furthermore, we can arrange these paths in some order such that for each
path after the first, the corresponding green edge in $H$ has an endpoint
in common with a previous path, and thus $\bt$ is 
a path tree in $\PU_{u,v}(\Lmax)$ for some $u\le\sum_i r_ik_i$.
%with the additional information, for every path $p$ in the component, of
%where $\start(p)$ and $\eend(p)$ are in $\Ext(\bt)$.  
Hence, the assumption \eqref{good_unionpaths} implies that there 
are $O\bigpar{n\Lmaxx^{2v-2}\log g}$ possible 
choices for the paths $p_{i,j,\nu}$ corresponding to this green component.
Consequently, taking the product over all green components of $H$,
we have
\begin{align}\label{mh}
  M_H&= O\bigpar{(\log g)^{\gq_G(H)}\Lmaxx^{2v(H)-2\gq_G(H)}n^{\gq_G(H)}}
\\\notag&
=O\bigpar{(\log g)^{\gq_G(H)}g^{-v(H)+\gq_G(H)}n^{v(H)}}
.\end{align}

Moreover, we have seen that $\tpi((\bP(i,j))_{ij})$ in \eqref{ac} is the
probability that each red component lies in a single cycle of $\sig$.
This entails that the $v(H)$ vertices in $H$
lie in at most $\gq_R(H)$ different cycles of $\bgs$,
with at least two of the vertices in each cycle,
and thus \refL{LXD} shows that
\begin{align}\label{ad}
  \tpi((\bP(i,j))_{ij}) = O\bigpar{g^{\gq_R(H)}n^{-v(H)}}.
\end{align}
Consequently, the total contribution to \eqref{ac} for all sequences of
lists yielding a given $H$ is, by \eqref{mh} and \eqref{ad},
\begin{align}
  \label{ae}
&O\bigpar{(\log g)^{\gq_G(H)}g^{\gq_G(H)+\gq_R(H)-v(H)}} 
%\notag\\ &\qquad
%=O\bigpar{(\log g)^{2\gq_G(H)}(\gam)^{v(H)-2\gq_R(H))}g^{\gq_G(H)+\gq_R(H)-v(H)}}
.\end{align}
Since each green or red component has size at least 2, it follows that
$v(H)\ge 2\gq_G(H)$ and $v(H)\ge 2\gq_R(H)$, and thus $\gq_G(H)+\gq_R(H) \le v(H)$.
If we here have strict inequality, then, since 
$g\to\infty$, \eqref{ae} shows that the
contribution is $o(1)$ and may be ignored. (There is only a finite number of
possible $H$ to consider.)

Hence, it suffices to consider the case $\gq_G(H)=\gq_R(H)=v(H)/2$. 
This implies that all green or red components have size 2, and thus are
isolated edges. It follows that if two 
different lists $\bP(i_1,j_1)$ and $\bP(i_2,j_2)$ contain two paths
$p_{i_1,j_1,\mnu_1}$ and $p_{i_2,j_2,\mnu_2}$ 
that have a common endpoint, then these paths have to coincide
(up to orientation).
Furthermore, if they coincide, and have, say, the same orientation so
$\eend(p_{i_1,j_1,\mnu_1})=\eend(p_{i_2,j_2,\mnu_2})$, then the red edges
from that vertex have to coincide, so 
$\start(p_{i_1,j_1,\mnu_1+1})=\start(p_{i_2,j_2,\mnu_2+1})$.
It follows easily 
that the two lists $\bP(i_1,j_1)$ and $\bP(i_2,j_2)$ are equivalent
in the sense $\bP(i_1,j_1)\equiv\bP(i_2,j_2)$ defined above.
However, we have excluded this possibility, and this contradiction shows
that  all paths $p_{i,j,\mnu}$ in the lists have disjoint sets of endpoints 
$\Ext(p_{i,j,\mnu})$. 

Let  $\sw:=\sum_i r_i k_i$ be the total number of paths in the lists
in $(\bP(i,j))_{ij}$. 
We have proved that in~\eqref{ac}, the contribution of the terms where
the paths in $(\bP(i,j))_{ij}$ %_{1\le i\le  q,\ 1\le j\le r_i}$ 
do not have $2\sw$
distinct endpoints is $o(1)$. 
Hence we may now consider the case where these endpoints are distinct. 
The calculation is very similar to the one
performed in \cite{SJ358}, therefore we will omit some details.
First, the total number of sequences of lists 
$(\bP(i,j))_{ij}$ such that 
$\bP(i,j)\in \PPkmi(T_n)$ and no two $\bP(i,j)$ are equivalent 
is
\begin{align}\label{mg}
  \prod_{i=1}^q\prod_{j=1}^{r_i}\bigpar{\ppkmi(T_n)+O(1)}
\sim
  \prod_{i=1}^q{\ppkmi(T_n)}^{r_i},
\end{align}
which by \eqref{good_paths} is
\begin{align}\label{mig}
\prod_{i=1}^q  \Theta\Bigpar{\Bigpar{\frac{n^2}{g}}^{r_ik_i}}
=\Theta\Bigpar{\Bigpar{\frac{n^2}{g}}^{w}}.
\end{align}
If the endpoints of the lists are not distinct, then the construction above
yields a graph $H$ with $v(H)\le2\sw-1$.
For each such graph $H$, the number of such sequences of lists is by
\eqref{mh}, recalling $\gq_G(H)\le v(H)/2$,
\begin{align}\label{mh2}
  M_H&
%= O\bigpar{(\log g)^{\gq_G(H)}g^{-v(H)+\gq_G(H)}n^{v(H)}}
=
O\bigpar{(\log g)^{v(H)/2}g^{-v(H)/2}n^{v(H)}}
=
O\Bigpar{(\log g)^{v(H)/2}\Bigpar{\frac{n}{g\qq}}^{2\sw-1}}
\notag\\&
=o\Bigpar{\Bigpar{\frac{n}{g\qq}}^{2\sw}}
=o\Bigpar{\Bigpar{\frac{n^2}{g}}^\sw}
.\end{align}
Hence, comparing with \eqref{mig}, 
we see that the number of  sequences of lists where the endpoints are not
distinct is a fraction $o(1)$ of the total number.
In other words, the number of sequences of lists that have $2w$ distinct
endpoints is $1-o(1)$ times the total number in \eqref{mg}.
For each such sequence of lists $(\bP(i,j))_{ij}$, we have 
\begin{align}\label{mgg}
\tpi((\bP(i,j)_{ij})\sim\Bigpar{\frac{6g}{n^2}}^w  
=\prod_{i=1}^q \Bigpar{\frac{6g}{n^2}}^{r_ik_i}  
\end{align}
by \refLs{LXC} and \ref{LXD} (for the case that some cycle in $\bgs$ covers
more than one pair of endpoints).
Consequently, \eqref{ac}, \eqref{mg} and \eqref{mgg} yield
\begin{align}\label{acc}
  \E \prod_{i=1}^q \bigpar{\tcckmi(T_n,\bgs)}_{r_i} 
= \bigpar{1+o(1)}   \prod_{i=1}^q\biggpar{\frac{\ppkmi(T_n)}{2k_i}}^{r_i}
 \Bigpar{\frac{6g}{n^2}}^{r_ik_i}  +o(1),
\end{align}
and \eqref{lc3} follows, recalling \eqref{pia} and \eqref{lc2a}--\eqref{lc2b},
\end{proof}

\begin{lemma}\label{LPoi}
Let $(T_n)$ be a good sequence of trees.
Then, for every $m\ge0$ and $k\ge1$,
\begin{align}
  \cckm(T_n,\bgs)\dto \Poi\bigpar{\gL_k(m)}
\end{align}
as \ntoo. Moreover, this holds jointly for any (finite) number of pairs $(m,k)$,
with the limit Poisson variables  being  independent.
\end{lemma}
\begin{proof}
  \refL{Lmom} implies by the method of moments that
  $\tcckmi(T_n,\bgs)\dto\Poi\bigpar{\gL_{k_i}(m_i)}$ jointly, 
with independent limits,
for any set of pairs $(m_i,k_i)$.
Furthermore, \eqref{lc2b} implies that \whp{}
$\cckmi(T_n,\bgs)=\tcckmi(T_n,\bgs)$ 
for each $(m_i,k_i)$, and thus $\cckmi(T_n,\bgs)$ converge to the same limits.
\end{proof}

\subsection{Letting $M\to\infty$}\label{SSMoo}
We have so far kept $M$ fixed.
Now it is time to let $M\to\infty$.
We therefore add $M$ to the notations when necessary.

Recall $\gL_k(m)=\gL_k(m;M)$ defined in \eqref{lc2a}. We define also, for
integers $a$ and $b$ with $0\le a\le b<\infty$,
\begin{align}\label{gLLk}
  \gL_k[a,b; M]:=\sum_{m=a}^{b-1}\gL_k(m;M).
\end{align}
We begin by finding the asymptotics of this as $M\to\infty$.

\begin{lemma}\label{LgL}
Let $a(M)$ and $b(M)$ be integers depending on $M$ such that $a(M)\le b(M)$ 
and, as \Mtoo,
\begin{equation}\label{sax}
\frac{a(M)}{\Mmax}\to x
\end{equation}
and
\begin{equation}\label{sby}
\frac{b(M)}{\Mmax}\to y.
\end{equation}
Then, as \Mtoo,
for every fixed $k$,
\begin{equation}\label{llambda}
\Lambda_k[a(\Mmax),b(\Mmax);M]\to \lambda_k^{x,y}
:=\frac{y^{2k}-x^{2k}}{(2k)(2k)!}
=\int_x^y\frac{t^{2k-1}}{(2k)!}\dd t
.\end{equation}
\end{lemma}

\begin{proof}

We first note that
on the one hand, we have, if $\ell\ge k$,
\begin{align}\label{aw1}
\sum_{|\bm|= \ell\atop s(\m)=k} \Pm&\geq \sum_{|\bm|= \ell\atop s(\m)=k}\prod_{i=1}^k 2m_i
\notag\\
&=2^k[z^\ell]\left(\frac{z}{(1-z)^2}\right)^k\notag\\
&=2^k\binom{\ell+k-1}{\ell-k}\notag\\
&\geq 2^k \frac{(\ell-k)^{2k-1}}{(2k-1)!}
\end{align}
and, similarly,
\begin{align}\label{aw2}
\sum_{|\bm|= \ell\atop s(\m)=k} \Pm&
\leq \sum_{|\bm|= \ell\atop s(\m)=k}\prod_{i=1}^k (2m_i+2)
\notag\\
&\leq2^k\sum_{|\bm|= \ell+k\atop s(\m)=k}\prod_{i=1}^k m_i\notag\\
&=2^k\binom{\ell+2k-1}{\ell}\notag\\
&\leq 2^k \frac{(\ell+2k)^{2k-1}}{(2k-1)!}.
\end{align}
Combining \eqref{aw1} and \eqref{aw2}, we obtain
\begin{align}\label{aw3}
  \sum_{|\bm|= \ell\atop s(\m)=k}\Pm
=  \frac{2^{k}}{(2k-1)!} \ell^{2k-1} + O\bigpar{1+\ell^{2k-2}},
\end{align}
uniformly in $\ell\ge0$.
Hence, since $b(M)=O(M)$ by \eqref{sby},
\begin{align}\label{aw4}
  \sum_{a(M)\le |\bm|< b(M)\atop s(\m)=k}\Pm
&= \sum_{\ell=a(M)}^{b(M)} \frac{2^{k}}{(2k-1)!} \ell^{2k-1} 
+ O\bigpar{M^{2k-1}}
\notag\\&
=  \frac{2^{k}}{(2k-1)!} \frac{b(M)^{2k} -a(M)^{2k} }{2k}
+ O\bigpar{M^{2k-1}}
\end{align}
Consequently, \eqref{lc2a} and the assumptions \eqref{sax} and \eqref{sby} yield
\begin{align}\label{aw5}
&\Lambda_k[a(\Mmax),b(\Mmax);M]
=
\sum_{m=a(M)}^{b(M)-1}   \gL_k(m;M)
\notag\\&\qquad
=\frac{1}{2k}\left(\frac{1}{2\Mmax^2}\right)^{k}
\Bigpar{ \frac{2^{k}}{(2k-1)!} \frac{b(M)^{2k} -a(M)^{2k} }{2k}
+O\bigpar{M^{2k-1}}}
\notag\\&\qquad
\to\frac{y^{2k}-x^{2k}}{(2k)(2k)!}.
\end{align}
which completes the proof.
\end{proof}

As we have seen above, a cycle $\sC$ in 
the underlying graph of $(T_n,\bgs)$ 
is
given by a list $\bP$ of paths satisfying \ref{C1}--\ref{C3}.
Define $s(\sC):=s(\bP)$, the number  of paths in the list.

\begin{lemma}\label{LPPk}
%  Suppose that \ntoo{} and that $g=g_n\to\infty$ with $g=o(n)$.
Let $(T_n)$ be a good sequence of trees, and
let $\fC_n$
be the set of simple cycles in 
the underlying graph of
$(T_n,\bgs)$.
Further, for $k\ge1$, let $\fC\kkk_n:=\set{\sC\in\fC_n:s(\sC)=k}$,
and consider the
(multi)set of their lengths 
$\Xi\kkk_n:=\bigset{|\sC|/\Lmax:\sC\in\fC\kkk_n}$, with $\Lmax$ given by
\eqref{ellen}. 
%=\bigset{(12g/n)\qq Z_i}$. 
Then the random set\/ $\Xi\kkk_n$,
regarded as a point process on $\ooo$, converges in distribution
to a Poisson process on $\ooo$ with intensity
$\xfrac{t^{2k-1}}{(2k)!}$. Moreover, this holds jointly for any finite
number of $k$.
\end{lemma}

\begin{proof}
Let $\cCxy_{kn}$ be the number of elements of $\Xi\kkk_n\cap[x,y)$, i.e.,
the number of simple cycles $\sC$ in the underlying graph of
$(T_n,\bgs)$ such that $s(\sC)=k$
and $x\Lmax\le |\sC|<y\Lmax$.
  The conclusion is equivalent to, 
with $\gl_k^{x,y}$ defined in \eqref{llambda},
  \begin{align}\label{lppk}
    \cCxy_{kn}\dto \Poi\bigpar{\gl_k^{x,y}},
\qquad\text{as \ntoo}
%\qquad\text{where}\quad \gl_k^{x,y}:=\int_x^y \frac{t^{2k-1}}{(2k)!}\dd t
 , \end{align}
%with $\gl_k^{x,y}$ defined in \eqref{llambda},
for every interval $[x,y)$ with $0\le x<y<\infty$,
with joint convergence (to independent limits) for any finite set of
disjoint such intervals; moreover, this is to hold jointly for several $k$.
We show that this follows from the similar statement in \refL{LPoi} by
letting $\Mmax\to\infty$. For notational convenience, we consider only a
single $k$ and a single interval $[x,y)$; the general case follows in the
same way.

Let 
\begin{align}
a^+(\Mmax)&=\bigpar{\floor{x\Mmax}-k}\vee 0,\\
a^-(\Mmax)&=\lceil x\Mmax\rceil,\\
b^-(\Mmax)&=\bigpar{\lfloor{y\Mmax}\rfloor-k}\vee0,\\
b^+(\Mmax)&=\lceil {y\Mmax}\rceil.
\end{align}
Consider only $M$ that are so large that $b^-(M)>a^-(M)$.
Then, it follows from \eqref{lp} that
\begin{align}\label{lp1}
C_{kMn}^-:=  \cckMabx-\le \cCxy_{kn}\le C_{kMn}^+:=\cckMabx+.
\end{align}
By \refL{LPoi}, 
for every fixed $M$, as \ntoo,
\begin{align}\label{lp2}
  C_{kMn}^- =\sum_{m=a^-(M)}^{b^{-}(M)-1}\cckm
\dto Z_{kM}:=\Poi\bigpar{\gL_k[a^-(M),b^-(M);M]}.
\end{align}
Moreover, by \refL{LgL}, as \Mtoo
\begin{align}\label{lp3}
  Z_{kM}\dto \Poi\bigpar{\gl_k(x,y)}.
\end{align}
Similarly, for every fixed $M$, as \ntoo, 
\begin{align}\label{lp4}
  C_{kMn}^+-C_{kMn}^- &=
\cckMx{a^+(M),a^-(M)}
+ \cckMx{b^-(M),b^+(M)}
\notag\\&
\dto W_M := \Poi\bigpar{\hgl(M)},
\end{align}
where
\begin{align}\label{lp5}
  \hgl(M):=\gL_k[a^+(M),a^-(M);M]+\gL_k[b^-(M),b^+(M);M],
\end{align}
and thus, by \refL{LgL} again, 
\begin{align}\label{lp6}
  \hgl(M)\to0\qquad\text{as }\Mtoo
.\end{align}
Consequently, using \eqref{lp1}, \eqref{lp4}
\begin{align}\label{lp7}
\limsup_{\ntoo}  \P\bigsqpar{\cC^{x,y}_{kn}\neq C^-_{kMn}}
\le  \limsup_{\ntoo} \P\bigsqpar{C^+_{kMn}\neq C^-_{kMn}}
= \P\bigpar{W_M>0},
\end{align}
and thus, by \eqref{lp6}, 
\begin{align}\label{lp8}
\limsup_{\Mtoo}\limsup_{\ntoo}  \P\bigsqpar{\cC^{x,y}_{kn}\neq C^-_{kMn}}
\le  \lim_{\Mtoo} \P\bigpar{W_M>0} 
=0.
\end{align}

Finally, \eqref{lp2}, \eqref{lp3} and \eqref{lp8} imply
\eqref{lppk}
%\begin{align}\label{lp9}
%  \cC^{x,y}_{kn}\dto \Poi\bigpar{\gl(x,y)},
%\qquad\text{as \ntoo}
%\end{align}
by \cite[Theorem 4.2]{Billingsley}.% (or \cite[Theorem 4.28]{Kallenberg}).
\end{proof}

\subsection{Finishing the proof for good trees}\label{SSfinish}

Let $\cCxy_{kn}=\cCxy_{kn}(T_n,\bgs)$ be as in the proof of \refL{LPPk}, i.e.,
the number of simple cycles $\sC$ in the underlying graph of 
$(T_n,\bgs)$ such that $s(\sC)=k$
and $|\sC|/\Lmax\in[x,y)$.

\begin{lemma}\label{Lbig}
Let $(T_n)$ be a good sequence of trees.
Then, for every $\xmax<\infty$,
there exist $K$ and $N$ such that if $n>N$ and $k>K$, then
\begin{align}\label{lbig}
  \E \cCoy_{kn} < 2^{-k}.
\end{align}
\end{lemma}

\begin{proof}
We may assume $\xmax\ge1$.
Fix $M=\ceil{20000\xmax}$.
By \eqref{lp1},
\begin{align}\label{lb1}
  \cC_{kn}^{0,\xmax}\le C_{k}^{[0,\ceil{\xmax M};M]}(T_n,\bgs)
=\sum_{m<\xmax M}\cckm(T_n,\bgs).
\end{align}
We have, similarly as in \eqref{lc2c},
\begin{align}\label{hw5}
\E\bigsqpar{\cckm(T_n,\sig)}
=\frac{1}{2k}\sum_{\bP\in\PPkm(T_n)}{\pi(\bP)}.
\end{align}
Furthermore, arguing as in the proof of \refL{Lexpectation},
if \ref{C2} and \ref{C3} hold for some $\bP$ with $s(\bP)=k$, then 
$\cE_{n,g}^{(k)}$ holds (up to a relabelling), and thus by \refL{LXC},
\begin{align}\label{hw4}
  \pi(\bP)
\le
\P\bigpar{\cE_{n,g}^{(k)}}
\le \Bigpar{\frac{\CXC g}{n^2}}^k.
\end{align}
(We may assume that $g/n<1/7$ for $n>N$.)
Consequently, \eqref{hw5} yields
\begin{align}\label{hw6}
\E\bigsqpar{\cckm(T_n,\sig)}&
\le 
\bigabs{\PPkm(T_n)}\Bigpar{\frac{\CXC g}{n^2}}^k =
{\ppkm(T_n)}\Bigpar{\frac{\CXC g}{n^2}}^k 
\notag\\&
=
\Bigpar{\frac{\CXC g}{n^2}}^k\sum_{|\bm|=m,s(\bm)=k}{\pM(T_n)}
.\end{align}

As a special case of \eqref{good_paths} (with $s(\bm)=1$), we have for every
fixed $m$,
\begin{align}\label{lb2}
  P_m(T_n) \sim 
%\frac{n^2}{12 M^2 g}\Pm = 
\frac{n^2}{12 M^2 g}(2m+1).
\end{align}
Hence, there exists $N$ such that
\begin{align}\label{hw1}
  P_m(T_n) \le \frac{n^2}{M^2 g}(m+1)
\end{align}
for every $0\le m<\xmax M$ and $n>N$.

Consider only $n>N$. Then, \eqref{hw1} implies that for every
$\bm=(m_1,\dots,m_k)$ with $|\bm|< \xmax M$ and $s(\bm)=k$,
\begin{align}\label{hw2}
\pM(T_n)&\leq 
\prod_{i=1}^k P_{m_i}(T_n)
\le
\lrpar{\frac{ n^2}{M^2 g}}^k
\prod_{i=1}^{k}\bigpar{m_i+1}
%\notag\\&
%\le\lrpar{\frac{ n^2}{M^2 g}}^k
%\left(\frac{|\m|+k}{k}\right)^k
%\le\lrpar{\frac{2 n^2}{M^2 g}}^k
.\end{align}
Hence, by the arithmetic-geometric inequality, if also $k\ge K:=yM$,
\begin{align}\label{hw3}
\pM(T_n)& 
%\notag\\&
\le\lrpar{\frac{ n^2}{M^2 g}}^k
\left(\frac{|\m|+k}{k}\right)^k
\le\lrpar{\frac{2 n^2}{M^2 g}}^k
.\end{align}

Consequently, for $k\ge K$ and $n>N$,
\eqref{lb1} and \eqref{hw6} yield
\begin{align}\label{hw7}
\E\cCoy_{kn} %(T_n,\bgs)
\le 
\Bigpar{\frac{\CXC g}{n^2}}^k
\sum_{|\bm|< \xmax\Mmax,s(\bm)=k}\pM(T_n).
\end{align}
There are less than $(\xmax\Mmax+1)^k$ lists $\m$ with $|\m|< \xmax\Mmax$ and
$s(\m)=k$. Hence,  \eqref{hw7} and \eqref{hw3} yield
\begin{align}\label{hw8}
\E\cCoy_{kn} %(T_n,\bgs)
\le 
\lrpar{\frac{\CXC g}{n^2}}^k
\xpar{\xmax\Mmax+1}^k
\lrpar{\frac{2 n^2}{M^2 g}}^k
\le  
\lrpar{\frac{4\xmax\CXC}{M}}^k,
\end{align}
and the result \eqref{lbig} follows by our choice of $\Mmax$.
\end{proof}

\begin{proposition}\label{PPP}
  Let $(T_n)$ be a good sequence of trees, and
let as in \refL{LPPk} $\fC_n$
be the set of simple cycles in 
the underlying graph of
$(T_n,\bgs)$.
Consider the
(multi)set of the cycle lengths 
$\Xi_n:=\bigset{|\sC|/\Lmax:\sC\in\fC_n}$, with $\Lmax$ given by
\eqref{ellen}. 
%=\bigset{(12g/n)\qq Z_i}$. 
Then the random set\/ $\Xi_n$,
regarded as a point process on $\ooo$, converges in distribution
to a Poisson process $\Xi$ on $\ooo$ with intensity
$\xpfrac{\cosh(t)-1}t$.
\end{proposition}

\begin{proof}
Let, recalling \eqref{llambda},
\begin{align}\label{gl99}
  \gl^{x,y}:=\sumk \gl_k^{x,y}=\int_x^y\sumk \frac{t^{2k-1}}{(2k)!}\dd t
=\int_x^y\frac{\cosh t-1}{t}\dd t.
\end{align}
Then,  the conclusion is equivalent to
  \begin{align}\label{eleon}
    \cCxy_n\dto \Poi\bigpar{\gl^{x,y}},
\qquad\text{as \ntoo}
 , \end{align}
for every interval $[x,y)$ with $0\le x<y<\infty$,
with joint convergence (to independent limits) for any finite set of
disjoint such intervals.
As in the proof of \refL{LPPk}, for notational convenience, we consider only
a single interval $[x,y)$; the general case follows in the same way.

We have $\cCxy_n=\sumk\cCxy_{kn}$. Define, for $K\ge1$,
\begin{align}
  \cCxy_{\le K.n}:=\sum_{k=1}^K\cCxy_{kn}.
\end{align}
For each fixed $K$, \refL{LPPk} implies that, as \ntoo,
\begin{align}
  \cCxy_{\le K,n}\dto\Poi\bigpar{\gl_{\le K}^{x,y}},
\end{align}
with 
\begin{align}
  \gl_{\le K}^{x,y}:=\sum_{k=1}^K   \gl_{k}^{x,y}.
\end{align}
Hence, as $K\to\infty$, $\gl_{\le K}^{x,y}\to\gl^{x,y}$, and thus
\begin{align}
  \Poi\bigpar{\gl_{\le K}^{x,y}}
\dto
\Poi\bigpar{\gl^{x,y}}.
\end{align}
Moreover, by \refL{Lbig},  for large enough $K$,
\begin{align}
  \limsup_{\ntoo}\P(\cCxy_{\le K,n}\neq \cCxy_n)
\le \limsup_{\ntoo}\sum_{k>K} \P\bigsqpar{\cC_{kn}^{0,\xmax}>0}
\le \sum_{k>K}2^{-k}=2^{-K},
\end{align}
which tends to 0 as $K\to\infty$.
Hence, 
by \cite[Theorem 4.2]{Billingsley} %(or \cite[Theorem 4.28]{Kallenberg})
again, 
\eqref{eleon} follows.
\end{proof}

\subsection{Finishing the proof for random trees and maps}
\label{SSfinishR}

\begin{lemma}
  \label{Lgood}
For every fixed $\Mmax$ and\/  $\m\in \setM$,
\begin{equation}\label{lgood1}
\biggpar{\frac{\Mmax^2}{n\Lmax^2}}^{s(\m)}
\pM(\T_n)
\pto\Pm,
\end{equation}
and for every fixed $q$,  $w$ and $c$,
\begin{equation}\label{lgood2}
\frac{1}{ n \Lmaxx^{2w-2}\log g}\sum_{\bt\in\PU_{q,w}(c\Lmax)}N_\bt(\T_n)
\pto0.
\end{equation}
%In other words, \eqref{good_paths} and \eqref{good_unionpaths} hold in
%probability.
%
Moreover, we may 
%without loss of generality
assume that the random trees $\T_n$ are coupled such that the convergences
\eqref{lgood1} and \eqref{lgood2} hold a.s.,
and thus $(\T_n)$ is good a.s.
\end{lemma}
\begin{proof}
  First, \eqref{lgood1} is a consequence of \refLs{lem_cv_paths} and \refL{LpM},
and \eqref{lgood2} follows from \refL{lem_uniendp_sum}.

Consider the infinite vector $\bY$ of the \lhs{s} of \eqref{lgood1} and
\eqref{lgood2}, for all $\Mmax$, $\bm$, $q$, $w$ and integers $c$. Then
\eqref{lgood1}--\eqref{lgood2} say that $\bY$ converges in probability to
some non-random vector $\by$, in the product space $\bbR^\infty$.
By the Skorohod coupling theorem \cite[Theorem~4.30]{Kallenberg}, 
we may couple the random trees $\T_n$ such that $\bY\to\by$ a.s.,
and the result follows.
\end{proof}

\begin{proof}[Proof of \refTs{TPP} and \ref{thm_main}]
  By \refL{Lgood}, we may without loss of generality assume that the
  sequence of random graphs $\T_n$ is good a.s. 
Hence, if we condition on $(\T_n)$, we may apply  \refProp{PPP}.
Consequently, the conclusion of \refProp{PPP} holds also for the sequence of
random trees $\T_n$.
Moreover, the underlying graph $\GG(\T_n,\bgs)$ has the same distribution as the
random unicellular map $\u$,
and thus the result holds also for it.
\end{proof}

\subsection{Further remarks}

\begin{remark}\label{Rprimitive}
  We have, for definiteness, considered only simple cycles in this paper.
However, it follows from the proofs above, in particular the proof of
\refL{Lmom}, that whp{} all simple cycles in $\u$
with length $\le C$ are disjoint,
and thus every primitive cycle of length $\le C$ is simple.
(Recall that a primitive cycle may intersect itself, but it may not consist
of another cycle repeated several times.)
We omit the details.
\end{remark}

\begin{remark}\label{RPP}
It follows from the proofs above that 
the convergence in \refProp{PPP} holds jointly with the convergence for each
  fixed $k$ in \refL{LPPk}.
An alternative interpretation of this is 
that 
if
$\bbbN:=\set{1,2,\dots,\infty}$
is the one-point compactification of $\bbN$, then
the (multi)set of
points $\Xix_n:=\set{(|\sC|/\Lmax,s(\sC)):\sC\in\fC_n}$, regarded as a
point process in $\cS:=\ooo\times\bbbN$, converges to a certain Poisson
process $\Xix$ on $\cS$. 
(Cf.\ \eg{} \cite[Section 4]{SJ136} for the importance of using $\bbbN$
instead of $\bbN$ here.)
\end{remark}

This joint convergence to Poisson processes 
%in Lemma \ref{LPPk} and \refProp{PPP}
%(see also \refR{RPP}) 
implies, for example,
by standard arguments the following.
\begin{corollary}\label{C1k}
  Let $\sC_1$ be the shortest cycle in the underlying graph of 
$(\T_n,\bgs)$. (This is whp unique by \refT{TPP}.)
Then $s(\sC_1)$ has a limiting distribution, as \ntoo, given by
\begin{align}
  \P\bigpar{s(\sC_1)=k}
\to p_k:=\intoo 
\frac{z^{2k-1}}{(2k)!}\exp\Bigpar{-\int_0^z \frac{\cosh t-1}{t}\dd t} \dd z.
\end{align}
\end{corollary}

Numerically we have
$p_1\doteq0.792$, % 0.7915786263
$p_2\doteq0.177$, % 0.1770770301
$p_3\doteq 0.028$, % 0.02787602968
$p_4\doteq 0.003$. % 0.003174302453
% pp[5] := 0.0002742466001

However, as far as we know,  $s(\sC)$ has no natural interpretation for
cycles in the unicellular map.

\newcommand\arxiv[1]{\texttt{arXiv}:#1}

\end{document}